\documentclass[11pt,a4paper]{article}
\usepackage[latin1]{inputenc}
\usepackage[top=1.25in, bottom=1.25in, left=1.00in, right=1.00in]{geometry}

\usepackage{paralist}
\usepackage{color}
\usepackage{tikz}
\usetikzlibrary{shapes,arrows}

\usepackage{graphicx}
\usepackage{caption}
\usepackage{subcaption}

\usepackage{amsmath,empheq}
\usepackage{amsfonts}
\usepackage{amssymb}
\usepackage{cancel}
\usepackage{bm}
\usepackage{mathrsfs}
\usepackage{setspace}

\newtheorem{proposition}{Proposition}[section]
\newtheorem{result}{Result}[section]

\newenvironment{proof}[1][Proof]{\begin{trivlist}
\item[\hskip \labelsep {\bfseries #1}]}{\end{trivlist}}

\newenvironment{remark}[1][Remark]{\begin{trivlist}
\item[\hskip \labelsep {\bfseries #1}]}{\end{trivlist}}

\newcommand{\qed}{\nobreak \ifvmode \relax \else
      \ifdim\lastskip<1.5em \hskip-\lastskip
      \hskip1.5em plus0em minus0.5em \fi \nobreak
      \vrule height0.75em width0.5em depth0.25em\fi}

\numberwithin{equation}{section}
\numberwithin{table}{section}
\numberwithin{figure}{section}

\DeclareMathOperator*{\mydef}{\overset{def}{=}}
\DeclarePairedDelimiter\abs{\lvert}{\rvert}

\newcommand{\myprime}[1]{#1'}
\newcommand{\myds}[1]{{#1}^d}

\newcommand{\fref}[1]{\figurename~\ref{#1}}
\newcommand{\sref}[1]{Section~\ref{#1}}
\newcommand{\pref}[1]{Proposition~\ref{#1}}

\usepackage[blocks]{authblk}

\author[1]{J. Jagalur-Mohan}  \author[1]{O. Sahni} \author[2]{A. Doostan} \author[1]{A. A. Oberai}
\affil[1]{\footnotesize{Department of Mechanical, Aerospace \& Nuclear Engineering, Rensselaer Polytechnic Institute, Troy, NY}}
\affil[2]{\footnotesize{Aerospace Engineering Sciences, University of Colorado, Boulder, CO}}

\usepackage[pdfauthor={Jagalur-Mohan},
            pdftitle={Variational multiscale analysis: the fine-scale Green's function for stochastic partial differential equations},
            pdfsubject={vms-uq-gpc},
            pdfkeywords={variational multiscale method, Uncertainty quantification, generalized polynomial chaos}]{hyperref}

\title{Variational multiscale analysis: the fine-scale Green's function for stochastic partial differential equations}
\date{}

\begin{document}
\maketitle \let\thefootnote\relax\footnote{Submitted to SIAM/ASA Journal on Uncertainty Quantification}
\begin{abstract}
We present the variational multiscale (VMS) method for partial differential equations (PDEs) with stochastic coefficients and source terms. We use it as a method for generating accurate coarse-scale solutions while accounting for the effect of the unresolved fine scales through a model term that contains a fine-scale stochastic Green's function. For a natural choice of an ``optimal'' coarse-scale solution and $L^2$-orthogonal stochastic basis functions, we demonstrate that the fine-scale stochastic Green's function is intimately linked to its deterministic counterpart. In particular, 
\begin{inparaenum}[(i)] 
\item we demonstrate that whenever the deterministic fine-scale function vanishes, the stochastic fine-scale function satisfies a weaker, and discrete notion of vanishing stochastic coefficients, and
\item derive an explicit formula for the fine-scale stochastic Green's function that only involves quantities needed to evaluate the fine-scale deterministic Green's function.
\end{inparaenum}
We present numerical results that support our claims about the physical support of the stochastic fine-scale function, and demonstrate the benefit of using the VMS method when the fine-scale Green's function is approximated by an easier to implement, element Green's function.
\end{abstract}

\section{Introduction}

The comprehensive design of engineering systems requires a sound understanding of the underlying physics, along with the recognition of different kinds of uncertainties influencing its response. Often, these uncertainties are intrinsic to the physical phenomena under consideration. It is necessary to characterize and study their propagation, facilitating quantitative predictions about the uncertainty of relevant output quantities of interest. Experimental methodologies in this regard may prove to be expensive, and impractical with increasing complexity. To that end, the development of computational techniques assumes greater importance and the field of uncertainty quantification (UQ) has come into its own. 

In the probabilistic framework, representation of the uncertainties using random variables and processes is the preliminary task in any UQ method. This is mostly accomplished by relying on experimental data. The objective of the numerical technique can then be stated as the task of determining the uncertainties of the output quantities of interest and possibly their dependence on the characterized random inputs. Monte Carlo simulation (MCS) and related sampling based methods are some of the earliest approaches used for the purpose of uncertainty propagation \cite{Fi96}. However, owing to their slow convergence their use is limited. In particular, in the case of  large scale problems, where each deterministic solve at a sample point in itself is computationally intensive, they may become inefficient.

In recent years methods based on generalized polynomial chaos (gPC) \cite{XiKa02,Xi09,Xi10}, a generalization of the classical polynomial chaos \cite{SpGh89,GhSp03}, have been used extensively in UQ. These involve expressing the stochastic solution in an orthogonal polynomial basis of the input random parameters. The type of orthogonal polynomial chosen is dictated by the stochastic nature of the random input. The solution is essentially a spectral representation in the random space, and exhibits fast convergence when it depends smoothly on the random parameters. Numerical schemes using gPC typically involve a Galerkin projection onto the finite dimensional subspace spanned by the gPC basis in order to compute the expansion coefficients. This is the stochastic Galerkin approach and generally results in a coupled system for all the gPC expansion coefficients. It is labeled as intrusive, when requiring modification of the existing deterministic solver.

On the other hand, stochastic collocation methods \cite{MaYo02,XiHe05,GaZa07_a,Xi07,BaNoTe07}, which only need repetitive realizations of the deterministic solver are an attractive alternative. They satisfy the governing equations at a discrete set of points in the random space, and in this sense are similar to MCS. However, the discrete set of points in the case of stochastic collocation methods are chosen strategically using polynomial approximation theory to enhance convergence \cite{Xi10,LeKn10}. 

If the input uncertainties are characterized by a large number of random variables, then in a stochastic collocation method the set of discrete points becomes prohibitively large, thus imposing a severe computational burden. Similarly, in the context of stochastic Galerkin approaches, due to the tensor product construction of multi-dimensional bases the number of degrees of freedom grow exponentially with the number of random variables. In either case this exponential growth in computational complexity with respect to the number of input random variables  is termed as the {\it curse-of-dimensionality}. Several advances have been reported that counter this difficulty. In particular, multi-scale model reduction techniques when using stochastic Galerkin approaches have been investigated in \cite{DoEtal07} and methods exploiting low-rank or sparsity structures of quantities of interest have been presented in \cite{DoOw11,DoVaIa13,Xiu05a,Nobile08a,Doostan09,Nouy10,Hackbusch12}.

In this work, we achieve a reduction in computational complexity by developing a stochastic variational method that is designed to provide accurate solutions while using a coarse discretization in both physical and stochastic space. We recognize that by limiting our discretization to fewer elements and low-order gPC expansions, there will remain unresolved physical and stochastic features in the so called fine-scale space. However, we account for the effect of these features on the coarse-scale terms by introducing model terms. In that sense, our approach may be thought of as subgrid-modeling in the joint physical-stochastic space. We accomplish this within the framework of the variational multiscale (VMS) method \cite{Hu95,HuEtal98,HuSa07}. The VMS method provides a consistent framework for introducing such models in a variational problem, and has been extensively and effectively applied to problems in diverse areas such as compressible and incompressible fluid mechanics \cite{HuObMa01}, linear and non-linear elasticity \cite{GaHu98,KlMaSh99},  acoustics and electro-magnetics \cite{ObPi98,JaFeOb13}, albeit predominantly in a deterministic setting.

The starting point of the VMS method is the definition of a coarse-scale space and a projector $\mathcal{P}$, that maps functions into this space. This projector defines the desired coarse-scale solution and is a key part of the VMS formulation. The net result is a variational formulation posed on the coarse-scale space whose solution is guaranteed to be the continuous solution projected on to the coarse scales through the user-defined projector. This is accomplished by appending to the original weak formulation a term that captures the effect of the scales that live outside the coarse-scale space. This term comprises of the so called fine-scale Green's operator acting on the coarse-scale residual. The fine-scale Green's operator completely defines the VMS method.

In \cite{HuSa07} the authors have derived an expression for the abstract fine-scale Green's operator for variational problems posed on standard Hilbert spaces. Thereafter, for the case of deterministic PDEs, they have examined the locality of the fine-scale Green's function. The Green's function for most problems is typically non-zero across the whole physical domain, but this is not the case for the  fine-scale Green's function. It depends primarily on the projector used to construct the VMS method, and with appropriate choices, the fine-scale Green's function can be confined to smaller regions.  The locality plays an important role in designing simple operators that approximate the effect of the fine-scale Green's function, and hence make the VMS formulation a viable numerical method.

In this manuscript we extend and apply this analysis to stochastic PDEs.  We note that in the context of stochastic PDEs there have been previous studies that use the VMS method \cite{BaZa05,BaZa05b,BaZa06,GaZa07_b}, however, an explicit analysis of the fine-scale stochastic Green's operator has not been performed to the best of our knowledge. The main results of our analysis, which is performed assuming that $\mathcal{P}$ is defined by the inner-product associated with the trial-solution space and that the stochastic basis functions are $L_2$-orthogonal, are:

\begin{enumerate}
\item A statement about the locality (in physical space) of the fine-scale stochastic Green's operator. In particular we note that at the locations where the fine-scale deterministic Green's operator vanishes its stochastic counterpart need not vanish. However, at these locations when the fine-scale function is expanded in terms of the stochastic basis functions, the coefficients associated with the coarse-scale stochastic space will vanish.
\item An explicit expression for the fine-scale stochastic Green's function purely in terms of quantities required to construct the corresponding fine-scale deterministic Green's function. Implying thereby, that whenever the fine-scale deterministic Green's function is known, its stochastic counterpart can also be constructed.
\end{enumerate}
The first result suggests that even though the fine-scale stochastic Green's operator is not local in the physical space, it may be approximated by one as long its deterministic counterpart is local. As a consequence, the implementation of the VMS method is simplified immensely. We test the effect of this approximation and the performance of the resulting VMS method by numerically solving the stochastic advection-diffusion problem.

An outline of the paper is as follows. In \sref{sec:abstract_formulation} we introduce the abstract problem and explain the necessary mathematical framework. In \sref{sec:vms_formulation} we apply the VMS formulation to the abstract problem and present the two main results of this paper that link the form and properties of the fine-scale stochastic Green's function to its  deterministic counterpart. In \sref{sec:ade_eg} we consider the advection diffusion problem posed on one physical dimension and explicitly evaluate and examine the fine-scale stochastic Green's function. In \sref{subsec:numerical_results} we report numerical examples that test the performance of the VMS method when the fine-scale stochastic Green's function is approximated by the element Green's function for problems with single and multiple random variables. Finally in \sref{sec:conclusions} we end with concluding remarks.

\section{The abstract formulation} \label{sec:abstract_formulation}

\subsection{The abstract problem} \label{subsec:strong_form}

Consider a stochastic linear boundary value problem, wherein we seek a stochastic function $u(x,\omega): \mathcal{\bar{D}} \times \Omega \rightarrow \mathbb{R}$, such that the following equation holds almost surely in $\Omega$,
\begin{equation}  \label{eqn:strong_form}
     \left. \begin{aligned}
         \mathcal{L} u(x,\omega) &=f(x,\omega), \qquad & x \in \mathcal{D}, \\
         u(x,\omega) &= 0, \qquad & x \in \partial \mathcal{D}.
      \end{aligned}  \right\} 
\end{equation}
$\mathcal{L}$ is a stochastic linear  partial differential operator whose coefficients are one of the sources of uncertainty in the problem, $f(x,\omega)$ is a prescribed stochastic source term, $\mathcal{D}$ is the physical domain of the problem, and $\Omega$ is the set of elementary events ($\omega \in \Omega$).

\subsubsection{Probabilistic framework}
 For computational purposes, we will assume that the uncertainty has been parametrized using a finite set of independent real-valued random variables $ \{ \xi_i (\omega) \}_{i=1}^{q}$. Here $q$ is the number of coordinates that sufficiently characterizes the stochastic nature of the problem, and will henceforth be referred to as the stochastic dimension. $\xi(\omega) = \left(\xi_1(\omega),\cdots,\xi_q(\omega)\right)$ is a $q$-variate random vector in a properly defined probability space $\left(\Omega,\mathscr{F},\mathscr{P}\right)$, where $\mathscr{F}$ is the minimal $\sigma$-algebra of the subsets of $\Omega$, and $\mathscr{P}$ is the probability measure on $\mathscr{F}$. 
 
We know that each random variable $\xi_i$ is a map from the probability space $\left(\Omega,\mathscr{F},\mathscr{P}\right)$ to either a subset or whole of the real line. Let $\Upsilon_i \mydef \xi_i(\Omega) \subseteq \mathbb{R}$ denote the image of $\xi_i$, and $\rho_i: \Upsilon_i \rightarrow \mathbb{R}_{>0}$ its associated probability density function (PDF). The joint probability density of the random vector $\xi$ is
\begin{equation}
	 {\rho}(\xi) = \prod_{i=1}^{q} \rho_i(\xi_i),
\end{equation}
and has the support
\begin{equation}
	 {\Upsilon} = \prod_{i=1}^{q} \Upsilon_i  \subseteq \mathbb{R}^{q}.
\end{equation}
The solution to \eqref{eqn:strong_form} now admits a finite dimensional representation, {\it i.e.},
\begin{equation} 
	u(x,\xi) \mydef u\left(x,\xi_1(\omega),\cdots,\xi_q(\omega)\right) : \mathcal{\bar{D}} \times {\Upsilon} \rightarrow \mathbb{R}.
\end{equation}


\subsection{The variational formulation} \label{subsec:variational_formulation}

The appropriate space $\mathcal{V}$ to pose the variational problem is obtained by a tensor product of function spaces constructed over $\mathcal{D}$ and $\Omega$. Let $\myds{\mathcal{V}} \mydef H^1_0(\mathcal{D})$ be the first order Hilbert space consisting of functions that vanish on the boundary $\partial \mathcal{D}$, and $L^2(\Upsilon)$ be the space of random variables with finite variance with respect to the density function ${\rho}(\xi)$. The space $\mathcal{V} \mydef \myds{\mathcal{V}}   \otimes L^2(\Upsilon)$ is the linear span of elements of the form $\nu \Psi$, where $ \nu \in   \myds{\mathcal{V}} $ and  $\Psi \in L^2(\Upsilon)$, {\it i.e.},
 \begin{equation} \mathcal{V} =  \Bigl\{ v  \Big| v = \nu \Psi  ,   \forall \nu \in   \myds{\mathcal{V}} ,  \forall \Psi \in L^2(\Upsilon) \Bigr\}. \end{equation} 
The space $\mathcal{V}$ is equipped with an inner-product inherited from the respective individual function spaces. The inner-product associated with $  \myds{\mathcal{V}} \mydef H^1_0(\mathcal{D})$ is
\begin{equation} \label{eqn:det_inner_product_def}
 {(\nu,\tilde{\nu})}_{ \myds{\mathcal{V}}} \mydef \int_{x \in \mathcal{D}} \nabla \nu \cdot \nabla \tilde{\nu}  dx,   \qquad \forall \nu, \tilde{\nu} \in \myds{\mathcal{V}},
 \end{equation}
and the inner-product associated with  $L^2(\Upsilon)$ is
\begin{equation}
  {(\Psi,\tilde{\Psi})}_{L^2(\Omega)} \mydef \mathrm{E} \bigl[ \Psi \tilde{\Psi} \bigr] = \int_{\xi \in \Upsilon} \Psi \tilde{\Psi} \rho d \xi,  \qquad \forall \Psi,\tilde{\Psi} \in L^2(\Upsilon),
 \end{equation} 
where  $\mathrm{E}$ is the mathematical expectation operator.  The inner-product $\left( \cdot, \cdot \right)_{\mathcal{V}} : \mathcal{V} \times \mathcal{V} \rightarrow \mathbb{R}$ is defined as
\begin{equation} \label{eqn:inner_product_def}
   \left( v , \tilde{v} \right)_{\mathcal{V}} \mydef  \mathrm{E} \Bigl[ \int_{x \in \mathcal{D}} \nabla v \cdot \nabla \tilde{v}  dx \Bigr] = \mathrm{E} \Bigl[   {(v , \tilde{v})}_{ \myds{\mathcal{V}}} \Bigr], \qquad \forall v , \tilde{v}  \in \mathcal{V}.
\end{equation}

Let $\mathcal{V}^{*}$ be the dual of $\mathcal{V}$, and $\prescript{}{\mathcal{V}^{*}}{\bigl<}  \cdot , \cdot {\bigr>}_{\mathcal{V}} $ be the pairing between them. We define the bilinear form associated with the dual pair $\prescript{}{\mathcal{V}^{*}}{\bigl<}  \cdot , \cdot {\bigr>}_{\mathcal{V}} : \mathcal{V}^{*} \times \mathcal{V} \rightarrow \mathbb{R}$  as
\begin{equation} \label{eqn:dual_pair_Def}
   \prescript{}{\mathcal{V}^{*}}{\bigl<}  \chi , v {\bigr>}_{\mathcal{V}} \mydef  \mathrm{E} \Bigl[ \int_{x \in \mathcal{D}} \chi v dx \Bigr], \qquad \forall \chi \in \mathcal{V}^{*}, \forall v \in \mathcal{V}.
\end{equation}
Using \eqref{eqn:dual_pair_Def},  the variational statement of \eqref{eqn:strong_form} is: find $u \in \mathcal{V}$ such that,
\begin{equation} \label{eqn:variational_formulation}
	  \prescript{}{\mathcal{V}^{*}}{\bigl<}  \mathcal{L} u , w {\bigr>}_{\mathcal{V}}  = \prescript{}{\mathcal{V}^{*}}{\bigl<}  f , w {\bigr>}_{\mathcal{V}},  \qquad \forall w \in \mathcal{V}.
\end{equation}

If \eqref{eqn:strong_form} were posed in a deterministic context, $\myds{\mathcal{V}}$ would be the relevant space for its corresponding variational problem. Denoting the dual of $\myds{\mathcal{V}}$ by  ${\myds{\mathcal{V}}}^{*}$, we establish a pairing between them, $\prescript{}{ {\myds{\mathcal{V}}}^{*} }{\bigl<}  \cdot , \cdot {\bigr>}_{\myds{\mathcal{V}}}$, with the associated bilinear form $\prescript{}{ {\myds{\mathcal{V}}}^{*} }{\bigl<}  \cdot , \cdot {\bigr>}_{\myds{\mathcal{V}}} : {\myds{\mathcal{V}}}^{*} \times \myds{\mathcal{V}} \rightarrow \mathbb{R}$ defined as,
\begin{equation} \label{eqn:deterministic_dual_pair_Def}
   \prescript{}{ {\myds{\mathcal{V}}}^{*} }{\bigl<}  \lambda , \nu {\bigr>}_{\myds{\mathcal{V}}}  \mydef \int_{x \in \mathcal{D}} \lambda \nu dx, \qquad \forall \lambda \in \mathcal{V}^{d*}, \forall \nu \in \mathcal{V}^d.
\end{equation}

\subsubsection{Numerical approach}

\paragraph{Spectral stochastic discretization} 
In the context of the spectral stochastic methods, the solution $u(x,\xi)$ of \eqref{eqn:variational_formulation} is represented by an infinite series of the form
\begin{equation} \label{eqn:infinite_spectral_represenation}
u(x,\xi) = \sum_{ m \in \mathbb{N}^q_0 } \widehat{u}_{m}(x) \Phi_{m}(\xi),
\end{equation}
where $ \mathbb{N}^q_0 \mydef \big \{ (m_1,\cdots,m_q) : m_i \in \mathbb{N} \cup \{ 0 \} \big \}$ is the set of multi-indices of size $q$ defined on non-negative integers. The set $\{\Phi_{m}(\xi)\}$ is a basis for $L^2(\Upsilon)$, and consists of multi-dimensional polynomials of random variables, referred to as generalized polynomial chaos (gPC). Let $\phi_{m_i}(\xi_i)$ denote the univariate polynomial of degree $m_i$ orthogonal with respect to $\rho_i(\xi_i)$. Each basis function $\Phi_{m}(\xi)$ is a tensor product of the univariate polynomials $\phi_{m_i}(\xi_i)$,
\begin{equation}
	\Phi_{m}(\xi) = \phi_{m_1}(\xi_1) \phi_{m_2}(\xi_2) \cdots  \phi_{m_q}(\xi_q),  \qquad m = \mathbb{N}^{q}_{0}.
\end{equation}
The total order of $\Phi_{m}(\xi)$ is $\abs{m} = \sum _{i=1}^{q} m_i$. It is assumed that the univariate polynomials are normalized, which leads to the orthonormality property,
\begin{equation} \label{eqn:gPC_orthonormal_property}
	\mathrm{E}\left[ \Phi_{m}(\xi)  \Phi_{n}(\xi) \right] = \int_{\xi \in \Upsilon} \Phi_{m}(\xi) \Phi_{n}(\xi)  {\rho}(\xi) d\xi = \delta_{mn}, \qquad \forall m,n \in \mathbb{N}^{q}_{0},
\end{equation}
where $ \delta_{mn}$ is the Kronecker delta function. Using \eqref{eqn:gPC_orthonormal_property}, we can interpret $\{ \widehat{u}_{m}(x) \}$ in \eqref{eqn:infinite_spectral_represenation} as generalized Fourier coefficients obtained by the projection of $u(x,\xi)$ onto each basis function $\Phi_{m}(\xi)$, {\it i.e.},
\begin{equation} \label{eqn:generalized_fourier_coefficient}
	\widehat{u}_{m}(x) = \mathrm{E}\left[ u(x,\xi) \Phi_{m}(\xi)  \right] = \int_{\xi \in \Upsilon} u(x,\xi) \Phi_{m}(\xi) {\rho}(\xi) d\xi, \qquad m \in \mathbb{N}^{q}_{0}.
\end{equation}

\paragraph{The Galerkin projection} The finite dimensional Galerkin approximation is obtained by restricting the search in \eqref{eqn:variational_formulation} to a finite dimensional subspace $\bar{\mathcal{V}} \subset \mathcal{V}$. It will consist of only a finite number of spectral modes from the series representation  \eqref{eqn:infinite_spectral_represenation}. We define $\bar{\mathcal{V}}\mydef\mathcal{X}^h\otimes\mathcal{Y}^p$, where $\mathcal{X}^h \subset \myds{\mathcal{V}}$ is the finite dimensional space of piecewise continuous polynomials $\left\{ \Theta_i(x) \right\}_{i=1,\cdots,N^d}$ defined on a partitioning of $\mathcal{D}$, and  $\mathcal{Y}^p \subset L^2(\Upsilon)$ is the $q$-variate orthogonal polynomial space of total degree at most $p$, {\it i.e.},
\begin{equation}
	\mathcal{Y}^p \mydef \mathrm{span} \Big \{ \Phi_{m}(\xi)  \Big| m \in \mathbb{N}^{q}_{0}, \abs{m} \leq p \Big \}.
\end{equation}
The dimensionality of the space $\mathcal{Y}^p$ is,
\begin{equation}
	N^s = \frac{(q+p)!}{q!p!},
\end{equation}
which increases exponentially as a function of the dimension $q$. The space $\mathcal{X}^h$ with dimensionality $N^d$, will also be referred to by $\myds{\bar{\mathcal{V}}}$, a finite dimensional subspace of $\myds{\mathcal{V}}$.
Any $\bar{v} \in \bar{\mathcal{V}}$ can be written as,
\begin{equation}
	 \bar{v} =  \sum_{i=1}^{N^d} \sum_{\abs{m} = 0}^{ p } v_{i,m} \Theta_i(x) \Phi_m(\xi), \qquad m \in \mathbb{N}^{q}_{0}.
\end{equation}

The finite dimensional Galerkin approximation of the exact solution $u(x,\xi)$ is: find $\bar{u} \in \bar{\mathcal{V}} $ such that,
\begin{equation} \label{eqn:galerkin_approximation}
	 \prescript{}{\mathcal{V}^{*}}{\bigl<}  \mathcal{L} \bar{u} , \bar{w} {\bigr>}_{\mathcal{V}}  = \prescript{}{\mathcal{V}^{*}}{\bigl<}  f , \bar{w} {\bigr>}_{\mathcal{V}},  \qquad \forall \bar{w} \in \bar{\mathcal{V}}.
\end{equation}

\begin{remark}
\eqref{eqn:galerkin_approximation} is referred to as the spectral stochastic finite element formulation. A non-spectral discretization of the stochastic dimension leads to a similar system, and has been the subject of previous work \cite{DeBaOd01,LeMaitre04}.
\end{remark}

\section{The variational multiscale formulation} \label{sec:vms_formulation}
In this section we first review the concept of the variational multiscale formulation and the fine-scale Green's operator, and apply it to the stochastic PDE described in \sref{subsec:strong_form}. For this, we borrow ideas and closely follow the notation from \cite{HuSa07} and refer the interested reader to the same for a more elaborate discussion. Thereafter in \sref{subsec:projector}, for a specific choice of the projection  operator, we prove two propositions that relate the deterministic and stochastic VMS formulations. This leads us to the two main results of this manuscript. The first equates the vanishing of the deterministic fine scales at certain physical locations to the vanishing of the coarse-scale gPC coefficients of the stochastic fine scales at the same locations, and is presented in \sref{subsec:projector}. This provides insight into the locality of the fine-scale stochastic Green's operator, given the locality of its deterministic counterpart. The second result, which is presented in \sref{subsec:sgf}, provides an explicit expression for the fine-scale stochastic Green's function in terms of the corresponding deterministic quantities. Therefore, if the fine-scale Green's function for the deterministic problem has been determined, then the corresponding fine-scale Green's function for the stochastic problem naturally follows.

\subsection{The fine-scale Green's operator for the stochastic PDE}

As defined in the previous section, let $\bar{\mathcal{V}}$ be a closed subspace of $\mathcal{V}$, and  $\mathcal{P}$ be a linear projector onto $\bar{\mathcal{V}}$; thus, $ \mathcal{P}^{2} = \mathcal{P}$ and $ \mathrm{Range}(\mathcal{P}) = \bar{\mathcal{V}}$. We assume $\mathcal{P}$ to be continuous in $\mathcal{V}$, and further define  $\myprime{\mathcal{V}} = \mathrm{span}\{ v | v \in \mathcal{V}, \mathcal{P}v = 0 \}$.  $\myprime{\mathcal{V}}$ is indeed a closed subspace of $\mathcal{V}$, and we have the decomposition
\begin{equation}
	\mathcal{V}=\bar{\mathcal{V}}\oplus\myprime{\mathcal{V}}.
\end{equation}
From a practical perspective, $\bar{\mathcal{V}}$ is the space of computable coarse scales and $\myprime{\mathcal{V}}$ is the space of unresolved fine scales.  Any $v \in \mathcal{V}$ can be written uniquely as $v = \bar{v} + \myprime{v}$, where $\bar{v} \in \bar{\mathcal{V}}$ and $\myprime{v} \in \myprime{\mathcal{V}}$. The coarse and fine-scale components can be expressed in terms of the projector $\mathcal{P}$ as,
\begin{align}
 \label{eqn:vBar_def}   \bar{v}  &= \mathcal{P}v, \\
 \label{eqn:vPrime_def} \myprime{v} &= v -\mathcal{P}v.
\end{align}

In the standard variational problem \eqref{eqn:variational_formulation}, we seek the solution $u \in \mathcal{V}$. The solution can be split as $u = \bar{u} + \myprime{u}$. The objective of the VMS method is to construct a discrete variational problem whose solution $\bar{u} = \mathcal{P}u$. The variational formulation \eqref{eqn:variational_formulation} can be split by choosing the weighting function from the coarse and fine-scale spaces successively: find $\bar{u} \in \bar{\mathcal{V}}$, and $\myprime{u} \in \myprime{\mathcal{V}}$ such that,
\begin{align}
 \label{eqn:coarse_scale_prob}     \prescript{}{\mathcal{V}^{*}}{\bigl<} \mathcal{L}\bar{u},\bar{w}  {\bigr>}_{\mathcal{V}} +  \prescript{}{\mathcal{V}^{*}}{\bigl<} \mathcal{L} \myprime{u}, \bar{w} {\bigr>}_{\mathcal{V}} & =  \prescript{}{\mathcal{V}^{*}}{\bigl<} f, \bar{w} {\bigr>}_{\mathcal{V}}, \qquad \forall \bar{w} \in \bar{\mathcal{V}},\\
 \label{eqn:fine_scale_prob}      \prescript{}{\mathcal{V}^{*}}{\bigl<}  \mathcal{L}\bar{u},\myprime{w} {\bigr>}_{\mathcal{V}} +  \prescript{}{\mathcal{V}^{*}}{\bigl<} \mathcal{L} \myprime{u},\myprime{w} {\bigr>}_{\mathcal{V}} &=  \prescript{}{\mathcal{V}^{*}}{\bigl<} f,\myprime{w} {\bigr>}_{\mathcal{V}}, \qquad \forall \myprime{w} \in \myprime{\mathcal{V}}. 
\end{align}
We label \eqref{eqn:coarse_scale_prob} as the coarse-scale problem and assume that it is well posed for $\bar{u}$ alone; {\it i.e.}, given $\myprime{u}$ and $f$ we can determine a unique $\bar{u} \in \bar{\mathcal{V}}$. Analogously, \eqref{eqn:fine_scale_prob} is labeled as the fine-scale problem and is assumed to be well posed for $\myprime{u}$ alone.

In \eqref{eqn:fine_scale_prob}, the fine-scale Green's operator $\myprime{\mathcal{G}} : \mathcal{V}^{*} \rightarrow \myprime{\mathcal{V}}$ yields $\myprime{u}$ from the coarse-scale residual $ f -\mathcal{L}\bar{u}$, {\it i.e.},
\begin{equation}
	  \myprime{u} = \myprime{\mathcal{G}}( f -\mathcal{L}\bar{u}).
\end{equation}
Using $\myprime{\mathcal{G}}$ we can eliminate $\myprime{u}$ from \eqref{eqn:coarse_scale_prob} to obtain the VMS formulation for $\bar{u}$: find $\bar{u} \in \bar{\mathcal{V}}$ such that,
	\begin{equation}
		  \prescript{}{\mathcal{V}^{*}}{\bigl<} \mathcal{L}\bar{u},\bar{w}  {\bigr>}_{\mathcal{V}} -  \underbrace{ \prescript{}{\mathcal{V}^{*}}{\bigl<} \mathcal{L}\myprime{\mathcal{G}}\mathcal{L}\bar{u},\bar{w}  {\bigr>}_{\mathcal{V}}}_{VMS \quad contribution} =  \prescript{}{\mathcal{V}^{*}}{\bigl<} f, \bar{w} {\bigr>}_{\mathcal{V}} - \underbrace{ \prescript{}{\mathcal{V}^{*}}{\bigl<} \mathcal{L}\myprime{\mathcal{G}} f, \bar{w} {\bigr>}_{\mathcal{V}}}_{VMS \quad contribution}, \qquad \forall \bar{w} \in \bar{\mathcal{V}}. 
	\end{equation}

In \cite{HuSa07}, an expression for the fine-scale Green's operator $\myprime{\mathcal{G}}$ was derived, and  is given by,
\begin{equation}\label{eqn:fine_greens_infinite}
	\myprime{\mathcal{G}} = \mathcal{G} - \mathcal{G}\mathcal{P}^{T}(\mathcal{P}\mathcal{G}\mathcal{P}^{T})^{-1}\mathcal{P}\mathcal{G},
\end{equation}
where $\mathcal{P}^{T}: {\bar{\mathcal{V}}}^{*} \rightarrow \mathcal{V}^{*}$ denotes the adjoint of $\mathcal{P}$, and is defined as,
\begin{equation} \label{eqn:projector_adjoint_def}
			\prescript{}{\mathcal{V}^{*}}{\bigl<} \mathcal{P}^{T}\bar{\mu},v {\bigr>}_{\mathcal{V}}   =  \prescript{}{ \bar{\mathcal{V}}^{*} }{\bigl<} \bar{\mu},\mathcal{P}v {\bigr>}_{\bar{\mathcal{V}}},  \qquad  \forall v \in  \mathcal{V}, \bar{\mu} \in \bar{\mathcal{V}}^{*}.
\end{equation}
In \eqref{eqn:projector_adjoint_def} $\bar{\mathcal{V}}^{*}$ is the dual of $\bar{\mathcal{V}}$ and $\prescript{}{ \bar{\mathcal{V}}^{*} }{\bigl<} \cdot , \cdot  {\bigr>}_{\bar{\mathcal{V}}}$ is the pairing between them. Using \eqref{eqn:fine_greens_infinite} it is easy to show that,
\begin{equation} \label{eqn:fine_greens_infinite_properties}
	\myprime{\mathcal{G}}\mathcal{P}^{T} = 0, \qquad \mathrm{and} \qquad \mathcal{P}\myprime{\mathcal{G}}=0.
\end{equation}
 
In practical computational studies, the coarse space $\bar{\mathcal{V}}$ is a finite dimensional subspace of $\mathcal{V}$. For this case an alternate, and arguably more useful expression for the fine-scale Green's operator may be derived as follows. If $ N = \mathrm{dim}(\bar{\mathcal{V}})$, we can determine a set of linear functionals  $\{ \mu_i \}_{i=1 , \cdots , N}$ such that, $\forall v \in \mathcal{V}$,
\begin{equation} \label{eqn:mu_def}
	\prescript{}{\mathcal{V}^{*}}{\bigl<}  \mu_{i},v {\bigr>}_{\mathcal{V}} = 0, \quad  \forall i =  1, \cdots , N  \qquad \Leftrightarrow \qquad \mathcal{P}v = 0
\end{equation}
Using \eqref{eqn:mu_def} we can characterize $v$ as a fine-scale function, {\it i.e.} $v \in \myprime{\mathcal{V}}$. The set $\{ \mu_i \}_{i=1 , \cdots , N}$ can be inferred to be a basis for the image of $\mathcal{P}^{T}$ from \eqref{eqn:projector_adjoint_def}.  The discrete counterparts of the properties in \eqref{eqn:fine_greens_infinite_properties} are now given by,

\begin{equation} \label{eqn:fine_greens_finite_property_1}
	\myprime{\mathcal{G}} \mu_{i} = 0, \quad  \forall i =  1, \cdots , N ,
\end{equation}  
and
\begin{equation} \label{eqn:fine_greens_finite_property_2}
\prescript{}{\mathcal{V}^{*}}{\bigl<}   \mu_i,  \myprime{\mathcal{G}} \chi {\bigr>}_{\mathcal{V}} = 0, \quad \forall \chi \in \mathcal{V}^{*}, \forall i =  1, \cdots , N .
\end{equation}  
An expression for $\myprime{\mathcal{G}}$ for the finite dimensional case equivalent to \eqref{eqn:fine_greens_infinite}, can now be constructed in terms of the vector $\bm{\mu} \in (\mathcal{V}^{*})^{N}$,
\begin{equation} 
      \bm{\mu}^{T} =  \begin{bmatrix}   \mu_1  & \cdots &  \mu_N   \end{bmatrix} ,
\end{equation}
the vector $\mathcal{G}\bm{\mu}^{T}$,
\begin{equation} 
	\mathcal{G}\bm{\mu}^{T} =  \begin{bmatrix}   \mathcal{G} \mu_1  & \cdots & \mathcal{G} \mu_N  \end{bmatrix} ,
\end{equation}
the matrix $ \bm{\mu}\mathcal{G}\bm{\mu}^{T}  \in \mathbb{R}^{N \times N}$,
\begin{equation} 
 \bm{\mu}\mathcal{G}\bm{\mu}^{T} =  \begin{bmatrix}
  \prescript{}{\mathcal{V}^{*}}{\bigl<}   \mu_1,  \mathcal{G} \mu_1 {\bigr>}_{\mathcal{V}} & \cdots & \prescript{}{\mathcal{V}^{*}}{\bigl<}   \mu_1,  \mathcal{G} \mu_N {\bigr>}_{\mathcal{V}} \\
  \vdots  & \ddots & \vdots  \\
  \prescript{}{\mathcal{V}^{*}}{\bigl<}   \mu_N,  \mathcal{G} \mu_1 {\bigr>}_{\mathcal{V}}  & \cdots & \prescript{}{\mathcal{V}^{*}}{\bigl<}   \mu_N,  \mathcal{G} \mu_N {\bigr>}_{\mathcal{V}}
\end{bmatrix},  
\end{equation}
and the vector of functionals $\bm{\mu} \mathcal{G} : \mathcal{V}^{*} \rightarrow \mathbb{R}^{N}$ such that, $\forall \chi \in \mathcal{V}^{*}$
\begin{equation}  
 	\bm{\mu}\mathcal{G}(\chi) = 
	 \begin{bmatrix}
  	\prescript{}{\mathcal{V}^{*}}{\bigl<}   \mu_1,  \mathcal{G} \chi {\bigr>}_{\mathcal{V}} \\
  	\vdots  \\
  	\prescript{}{\mathcal{V}^{*}}{\bigl<}   \mu_N,  \mathcal{G} \chi {\bigr>}_{\mathcal{V}}
 \end{bmatrix}.
 \end{equation}
The expression for $\myprime{\mathcal{G}}$ in the finite dimensional case is
\begin{equation} \label{eqn:fine_greens_finite}
 \myprime{\mathcal{G}} = \mathcal{G} - \mathcal{G}\bm{\mu}^{T}[\bm{\mu}\mathcal{G}\bm{\mu}^{T}]^{-1}\bm{\mu}\mathcal{G} .
 \end{equation}
From \eqref{eqn:fine_greens_finite}, it is clear that in order to determine $\myprime{\mathcal{G}}$, we require an appropriately defined  Green's operator $\mathcal{G}$, and the set of linear functionals $\bm{\mu}$. In what follows, we demonstrate that if these are known for the deterministic problem, then their stochastic counterparts can also be constructed.


\subsection{An orthogonal projector for the stochastic problem} \label{subsec:projector}

The projector $\mathcal{P}$ is one of the key ingredients of the VMS formulation. This is because  the fine-scale Green's function is constructed to ensure that the solution to the coarse-scale problem is the function $\mathcal{P}u$. Given that the solution to \eqref{eqn:variational_formulation} resides in the Hilbert space $\mathcal{V}$, a natural choice for $\mathcal{P}$ (but not the only choice) is to select it to be the orthogonal projector associated with the inner-product $(\cdot, \cdot)_{\mathcal{V}}$ defined in \eqref{eqn:inner_product_def}. In this manuscript we make this choice, {\it i.e.} we select $\mathcal{P}$ such that,
\begin{equation} \label{eqn:projector_def}
			(\bar{w},\mathcal{P}u)_{\mathcal{V}}  =  (\bar{w} ,u )_{\mathcal{V}},  \qquad \forall \bar{w} \in \bar{\mathcal{V}}, \forall u \in {\mathcal{V}}.
\end{equation}
It will be useful to define a deterministic, or a physical-space component of $\mathcal{P}$, denoted by $\myds{\mathcal{P}} : \myds{\mathcal{V}} \rightarrow \myds{\bar{\mathcal{V}}}$, such that 
\begin{equation}\label{eqn:projector_det_def}
 {(\bar{\vartheta},\myds{\mathcal{P}} \nu)}_{\myds{\mathcal{V}}} =  {(\bar{\vartheta}, \nu)}_{\myds{\mathcal{V}}},  \qquad \forall \bar{\vartheta} \in \myds{\bar{\mathcal{V}}}, \forall \nu \in \myds{\mathcal{V}}.
 \end{equation}
 where the inner-product $(\cdot,\cdot)_{_{\myds{\mathcal{V}}}}$ is defined in \eqref{eqn:det_inner_product_def}.

In the remainder of this section we first establish a relation between $\mathcal{P}$ and $\myds{\mathcal{P}}$ in Proposition 3.1. Thereafter we use this to determine an explicit relation between the functionals used to characterize the stochastic and deterministic fine-scale spaces in Proposition 3.2. Finally, in Result 3.1 we use these to prove that, at certain physical locations when the deterministic fine scales vanish (and the coarse scales are exact), the stochastic fine scales have vanishing coarse-scale gPC coefficients.


\begin{proposition} \label{prop:projector}
Any element $v \in \mathcal{V}$ expressed using an infinite series of the form,
\begin{equation} \label{eqn:v_expansion}
v(x,\xi) = \sum_{\abs{m}=0}^{\infty} \widehat{v}_m(x) \Phi_m(\xi),
\end{equation}
admits the following series representation for its coarse-scale projection $\mathcal{P}v \in \bar{\mathcal{V}}$,
\begin{equation} \label{eqn:P_Pd}
\mathcal{P}v = \sum_{\abs{m}=0}^{p}  \myds{\mathcal{P}} \widehat{v}_m(x) \Phi_m(\xi),\end{equation}
where $\mathcal{P}$ and $\myds{\mathcal{P}}$ are defined in \eqref{eqn:projector_def} and \eqref{eqn:projector_det_def}, respectively.
\end{proposition}

\begin{proof}
Any element $\bar{w} \in \bar{\mathcal{V}} $ can be written as
\begin{equation}
\label{eqn:w_coarse} \bar{w} = \widehat{\bar{w}}_{k}(x) \Phi_m(\xi),
\end{equation}
where $\widehat{\bar{w}}_{k}(x) \in \myds{\mathcal{V}}$ and $ \Phi_m(\xi) \in L^2(\Upsilon)$.
Since $\mathcal{P}v \in \bar{\mathcal{V}} $ it admits the representation,
\begin{equation}
\label{eqn:pvcoarse_expansion} \mathcal{P}v = \sum_{\abs{n}=0}^{p} \widehat{\mathcal{P}v}_n(x) \Phi_n(\xi).
\end{equation}
Using \eqref{eqn:w_coarse},\eqref{eqn:pvcoarse_expansion} and \eqref{eqn:v_expansion} in \eqref{eqn:projector_def} we have,
\begin{equation} \label{eqn:ps_weak_form_1}
\Bigl( \widehat{\bar{w}}_{k}(x) \Phi_m(\xi) , \sum_{\abs{n}=0}^{p} \widehat{\mathcal{P}v}_n(x) \Phi_n(\xi)  \Bigr)_{\mathcal{V}}  = 
  \Bigl( \widehat{\bar{w}}_{k}(x) \Phi_m(\xi) , \sum_{\abs{l}=0}^{\infty} \widehat{v}_l(x) \Phi_l(\xi) \Bigr)_{\mathcal{V}} , \qquad  \forall \widehat{\bar{w}}_{k}(x) \Phi_m(\xi) \in \mathcal{V}.
\end{equation}
Using \eqref{eqn:inner_product_def} and the orthogonality of $\{ \Phi_m \}_{\abs{m}=0,\cdots,p}$ in \eqref{eqn:ps_weak_form_1} we have, 
\begin{equation} \label{eqn:ps_weak_form_2} 
\Bigl( \widehat{\bar{w}}_{k}(x),  \widehat{\mathcal{P}v}_m(x) \Bigr)_{\myds{\mathcal{V}}}  =  \Bigl( \widehat{\bar{w}}_{k}(x) , \widehat{v}_m(x) \Bigr)_{\myds{\mathcal{V}}} , \qquad  \forall \widehat{\bar{w}}_{k} \in \myds{\bar{\mathcal{V}}}.
\end{equation}
From \eqref{eqn:projector_det_def} the above equation implies $\widehat{\mathcal{P}v}_m(x) = \myds{\mathcal{P}} \widehat{v}_m(x)$.
\qed
\end{proof}


\begin{proposition}  \label{prop:mu}
Let $\{ \myds{\mu}_{i} \}_{i=1, \cdots , \myds{N}} \in {\myds{\mathcal{V}}}^{*} $ be a set of deterministic linear functionals  such that, $\forall \nu \in \myds{\mathcal{V}}$,
\begin{equation} \label{eqn:mud_def}
 \prescript{}{{\myds{\mathcal{V}}}^{*}}{\bigl<} \myds{\mu}_{i},\nu {\bigr>}_{\myds{\mathcal{V}}} = 0, \quad \forall i=1, \cdots , \myds{N} \qquad \Leftrightarrow \qquad \myds{\mathcal{P}}\nu = 0.
\end{equation}
We can then determine $ \{ {\mu}_{i,m} \}_{{\setstretch{0.75} \begin{array}{c@{\scriptstyle =}l} \scriptstyle i & \scriptstyle 1, \cdots , \myds{N} \\  \scriptstyle \abs{m} & \scriptstyle 0,\cdots,p \end{array}}} $,
\begin{equation} \label{eqn:mu_mud}
{\mu}_{i,m}(x,\xi) \mydef \myds{\mu}_{i}(x) \Phi_m(\xi),
\end{equation}
such that $\forall v \in \mathcal{V}$,
\begin{equation} \label{eqn:mu_tensor_def}
 \prescript{}{\mathcal{V}^{*}}{\bigl<} {\mu}_{i,m},v {\bigr>}_{\mathcal{V}}=0, \quad \forall \begin{array}{c@{=}l} i & 1, \cdots , \myds{N} \\  \abs{m} & 0,\cdots,p \end{array} \qquad \Leftrightarrow \qquad \mathcal{P}v=0.
\end{equation}\end{proposition}
In \eqref{eqn:mu_tensor_def} we have used a tensor notation to index the functionals $\mu$'s relevant for the stochastic problem, in contrast to \eqref{eqn:mu_def}, where they were first introduced using a single index. This change will be found convenient in elucidating the forthcoming results.

\begin{proof}

We first show $\mathcal{P}v = 0  \quad \Rightarrow \quad \prescript{}{\mathcal{V}^{*}}{\bigl<}  {\mu}_{i,m},v  {\bigr>}_{\mathcal{V}}=0$. Since $\mathcal{P}v = 0$, we conclude that $v \in \mathcal{V}'$, and from \eqref{eqn:vPrime_def}  we have,
\begin{equation} \label{eqn:vPrime_def_duplicate}  v = u - \mathcal{P}u, \end{equation}
where $ u \in\mathcal{V}$. Using the series expansions for $u$ and $\mathcal{P}u$ from \pref{prop:projector} in \eqref{eqn:vPrime_def_duplicate} we have,
\begin{align}
v(x,\xi) &= \sum_{\abs{n}=0}^{\infty} \widehat{u}_n(x) \Phi_n(\xi) - \sum_{\abs{n}=0}^{p}  \myds{\mathcal{P}} \widehat{u}_n(x) \Phi_n(\xi), \nonumber \\
\label{eqn:vprime_expanded_2} &= \sum_{\abs{n}=0}^{p} (\widehat{u}_n(x) - \myds{\mathcal{P}} \widehat{u}_n(x) )\Phi_n(\xi) + \sum_{\abs{n}= p+1}^{\infty} \widehat{u}_n(x) \Phi_n(\xi). 
\end{align}
From the above equation we conclude that for $\abs{n} \le p$,
\begin{equation}
\label{eq_prop1} \myds{\mathcal{P}} \widehat{v}_n(x) = \myds{\mathcal{P}}(\widehat{u}_n(x) - \myds{\mathcal{P}} \widehat{u}_n(x) )  = 0.
\end{equation}
We now consider $ \prescript{}{\mathcal{V}^{*}}{\bigl<}  {\mu}_{i,m},v {\bigr>}_{\mathcal{V}}$ . Using \eqref{eqn:mu_mud}  we have,
\begin{align}
\label{eqn:sdp_expanded_1} \prescript{}{\mathcal{V}^{*}}{\Bigl<}  {\mu}_{i,m},v {\Bigr>}_{\mathcal{V}}
&=  \prescript{}{\mathcal{V}^{*}}{\Bigl<}  \myds{\mu}_{i}(x) \Phi_m(\xi)  , \sum_{\abs{n}=0}^{\infty} \widehat{v}_n(x) \Phi_n(\xi) {\Bigr>}_{\mathcal{V}}, \\
\label{eqn:sdp_expanded_2} &= \sum_{\abs{n}=0}^{\infty} \mathrm{E} \bigl[  \Phi_m(\xi) \Phi_n(\xi) \bigr] \prescript{}{ {\myds{\mathcal{V}}}^{*} }{\bigl<}  \myds{\mu}_{i}(x)  ,  \widehat{v}_n(x)  {\bigr>}_{\myds{\mathcal{V}}} , \\
\label{eqn:sdp_expanded_3} &=\prescript{}{ {\myds{\mathcal{V}}}^{*} }{\bigl<}  \myds{\mu}_{i}(x)  ,  \widehat{v}_m(x)  {\bigr>}_{\myds{\mathcal{V}}}  = 0.
\end{align}
In the development above to derive \eqref{eqn:sdp_expanded_2} we have made use of the relation between the duality pairing for the stochastic and deterministic function spaces, to derive \eqref{eqn:sdp_expanded_3} we have made use of the orthonormality of the stochastic basis functions, \eqref{eq_prop1} and \eqref{eqn:mud_def}.

We now show  $\mathcal{P}v = 0 \Leftarrow \quad \prescript{}{\mathcal{V}^{*}}{\bigl<}  {\mu}_{i,m},v  {\bigr>}_{\mathcal{V}}=0$ to complete the proof. Using \eqref{eqn:mu_mud} and \eqref{eqn:v_expansion}, the condition $\prescript{}{\mathcal{V}^{*}}{\bigl<}  {\mu}_{i,m},v  {\bigr>}_{\mathcal{V}}=0$ can be written as,
\begin{align}
\label{eqn:mu_v_1} \prescript{}{\mathcal{V}^{*}}{\Bigl<} \myds{\mu}_{i}(x) \Phi_m(\xi),\sum_{\abs{n}=0}^{\infty} \widehat{v}_n(x) \Phi_n(\xi) {\Bigr>}_{\mathcal{V}} &=0,\\
\Rightarrow \qquad \sum_{\abs{n}=0}^{\infty} \mathrm{E} \bigl[  \Phi_m(\xi) \Phi_n(\xi) \bigr]
\prescript{}{ {\myds{\mathcal{V}}}^{*} }{\bigl<}  \myds{\mu}_{i}(x)  , 
\widehat{v}_n(x)  {\bigr>}_{\myds{\mathcal{V}}}& = 0 , \\
\label{eqn:mu_v_2} \Rightarrow \qquad \prescript{}{{\myds{\mathcal{V}}}^{*}}{\Bigl<}   \myds{\mu}_{i}(x)   , \widehat{v}_m(x)  {\Bigr>}_{\myds{\mathcal{V}}} & = 0.
\end{align}
From \eqref{eqn:mud_def} and  \eqref{eqn:mu_v_2} we conclude that $\myds{\mathcal{P}}\widehat{v}_m(x)=0$, which in conjunction with \pref{prop:projector} yields $\mathcal{P}v = 0$.
\qed
\end{proof}


\begin{result} \label{result:delta_mu}
Let any deterministic functional $\myds{\mu}_{i} $ be such that,
\begin{equation} {\mu}^{d}_{i}(x) = \delta (x -x_i), \end{equation}
where $x_i \in \mathcal{D}$.  The fine-scale Green's operator $\myprime{\mathcal{G}}$ then satisfies the property,
\begin{equation} \mathrm{E} \left[ \Phi_m(\xi)  \myprime{\mathcal{G}} (\chi)|_{x_i} \right] = 0, \qquad \forall \chi \in \mathcal{V}^{*}. \end{equation}
\end{result}

\begin{proof}

 From \eqref{eqn:fine_greens_finite_property_2} we know that the fine-scale Green's operator $\myprime{\mathcal{G}}$ is such that, 
\begin{equation}
\prescript{}{\mathcal{V}^{*}}{\bigl<}   \mu_{i,m},  \myprime{\mathcal{G}} \chi {\bigr>}_{\mathcal{V}} = 0, \quad \forall \chi \in \mathcal{V}^{*}, \quad \forall \begin{array}{c@{=}l} i & 1, \cdots , \myds{N} \\  \abs{m} & 0,\cdots,p \end{array}.
\end{equation}
Using \pref{prop:mu}, ${\mu}_{i,m}(x,\xi) = \delta (x -x_i) \Phi_m(\xi)$ and thus we have,
\begin{align}
	\prescript{}{\mathcal{V}^{*}}{\bigl<} \delta (x -x_i) \Phi_m(\xi), \myprime{\mathcal{G}} (\chi) {\bigr>}_{\mathcal{V}}
	& =   0, \\
\Rightarrow \mathrm{E} \left[ \Phi_m(\xi)  \myprime{\mathcal{G}} (\chi)|_{x_i} \right]	& = 0.
\end{align}
\qed 
\end{proof}

\subsubsection{A Remark about the locality of the fine-scale stochastic Green's operator}
In the case of linear deterministic problems posed on one spatial dimension, the VMS method using an $H^1$ projector yields nodally exact solutions \cite{Hu95,HuEtal98}. As a result the functionals $\myds{\mu}_i = \delta (x -x_i)$, where $x_i$ are the coordinates of the physical nodes, completely characterize the fine-scale solution. Result \ref{result:delta_mu} describes an attribute of the appropriately adapted VMS method for stochastic problems posed on one spatial dimension. The fine-scale solution is such that the expectation of its product taken against each member of the basis for the coarse space of random variables, vanishes at the physical nodes. Equivalently, if the fine-scale solution is expanded in terms of a gPC expansion, then the coefficients of this expansion corresponding to the coarse gPC modes will be zero.

Let us continue to assume that $\myds{\mu}_i = \delta (x -x_i)$, where $x_i$ are the coordinates of the physical nodes. Under these conditions as one refines in the stochastic space, the expectation of the product of the fine scales with an ever increasing set of the coarse space basis functions must vanish at the nodes. In the limit $p \to \infty$, this implies that the fine scales vanish at the nodes. Consequently, the fine-scale stochastic Green's operator tends to a Green's operator that is zero on the physical nodes. This operator is called the element Green's operator. A tremendous practical benefit of this observation is the fact that the element Green's operator is local in that its support is restricted to within a single element, and hence including it in the VMS formulation does not change the locality of the finite element method. Result 3.1 suggests that for the stochastic problem, using the element Green's operator instead of the fine-scale Green's operator might be a reasonable approximation. It is certainly a valid approximation in the limit of a very well refined stochastic space, {\it i.e.}, large enough $p$. This property is demonstrated in \sref{subsec:ade_fsgf_localization} through a concrete numerical example.

\subsection{The Green's function for stochastic problems} \label{subsec:sgf}
We begin this section by relating the stochastic Green's function to its deterministic counterpart. Thereafter, by making use of this relation and the propositions proved in the previous section, we derive an explicit expression for the fine-scale stochastic Green's function in terms of the deterministic Green's function and the linear functionals $\myds{\mu}_i$. Since these quantities are required in constructing the fine-scale deterministic Green's function, a natural conclusion is that one can construct the fine-scale stochastic Green's function whenever the corresponding fine-scale deterministic Green's function can be constructed.

The solution $u$ to \eqref{eqn:strong_form} can be formally expressed as $u=\mathcal{G}(f)$,  where $\mathcal{G}: \mathcal{V}^{*} \rightarrow \mathcal{V}$ is the Green's operator, {\it i.e.}, $\mathcal{G} = \mathcal{L}^{-1}$. The action of $\mathcal{G}$ on $f \in  \mathcal{V}^{*}$ is conveniently represented through the convolution of the Green's function $g: \mathcal{D} \times \mathcal{D} \times \Omega \times \Omega \rightarrow \mathbb{R}$ with $f$,
\begin{equation} \label{eqn:stochastic_greens_function}
	 u\bigl(x,\xi\bigr) = \int_{\tilde{x} \in \mathcal{D}} \int_{ \tilde{\xi} \in \Upsilon } g(x,\tilde{x} ; \xi, \tilde{\xi})   f(\tilde{x},\tilde{\xi}) d\tilde{x} d\tilde{\xi}.
\end{equation}
We label $g$ as the stochastic Green's function, and express it in terms of Green's function for deterministic problems. In this context, we interpret the deterministic Green's function as parametrized by the coefficients of the differential operator $\mathcal{L}$ in \eqref{eqn:strong_form} for various realizations of the random vector $\xi$, and use $\myds{g}_{\xi}$ to denote the same. We develop this concept in the ensuing discussion, and provide a definition for $g$.

Let $\beta(x,\xi)$ be a stochastic function, representing a typical coefficient of the differential operator $\mathcal{L}$ in \eqref{eqn:strong_form}, {\it e.g.}, see \sref{subsubsec:nRV_results} for an instance of  $\beta(x,\xi)$ in the case of advection-diffusion equation. For every possible realization $\tilde{\xi}$ of the random vector $\xi$, $\beta(x,\tilde{\xi})$, which is a function of the physical variable only,  has an associated deterministic differential operator $\myds{\mathcal{L}}_{\tilde{\xi}}$, and a dual term $\myds{f}_{\tilde{\xi}}(x) \mydef f(x,\tilde{\xi})$. We interpret $\myds{\mathcal{L}}_{\tilde{\xi}}$ as the instance of constructing $\mathcal{L}$ in \eqref{eqn:strong_form} with coefficients $\beta(x,\tilde{\xi})$. Together, $\myds{\mathcal{L}}_{\tilde{\xi}}$ and $\myds{f}_{\tilde{\xi}}(x)$ define the problem,
\begin{equation} \label{eqn:realization_strong_form}
	\myds{\mathcal{L}}_{\tilde{\xi}} \myds{u}_{\tilde{\xi}}(x) = \myds{f}_{\tilde{\xi}}(x).
\end{equation}
The solution to \eqref{eqn:realization_strong_form} can be obtained by using the Green's function  $\myds{g}_{\tilde{\xi}}: \mathcal{D} \times \mathcal{D} \rightarrow \mathbb{R}$, {\it i.e},
\begin{equation} \label{eqn:deterministic_greens_function}
	\myds{u}_{\tilde{\xi}}(x) = \int_{\tilde{x} \in \mathcal{D}} \myds{g}_{\tilde{\xi}}(x,\tilde{x})\myds{f}_{\tilde{\xi}}(\tilde{x}) d\tilde{x}.
\end{equation}
Moreover, $\myds{u}_{\tilde{\xi}}(x)$ is equivalent to evaluating the solution $u(x,\xi)$ to \eqref{eqn:strong_form} at the realization $\tilde{\xi}$. Thus an alternative way to write \eqref{eqn:deterministic_greens_function} having incorporated this property is
\begin{equation}
 u\bigl( x, \tilde{\xi} \bigr) = \int_{\tilde{x} \in \mathcal{D}} \myds{g}_{\tilde{\xi}}(x,\tilde{x}) f(\tilde{x},\tilde{\xi}) d\tilde{x}.
\end{equation}
The solution to \eqref{eqn:strong_form} can thus be written as
\begin{equation} \label{eqn:stochastic_greens_function_2}
u\bigl( x,\xi \bigr) = \int_{\tilde{x} \in \mathcal{D}} \int_{\tilde{\xi} \in \Upsilon } \myds{g}_{\xi}(x,\tilde{x}) \delta ( \xi - \tilde{\xi} ) f(\tilde{x},\tilde{\xi}) d\tilde{x} d\tilde{\xi}.
\end{equation}
The term inside the integral in \eqref{eqn:stochastic_greens_function_2} which convolves with $f(\tilde{x},\tilde{\xi})$ is defined as the stochastic Green's function, {\it i.e.},
\begin{equation} \label{eqn:g_gd}
 	g(x,\tilde{x} ; \xi , \tilde{\xi})  \mydef  \myds{g}_{\xi} (x,\tilde{x}) \delta ( \xi - \tilde{\xi} ).
\end{equation}

\subsubsection{The fine-scale stochastic Green's function} \label{subsubsec:fssgf}

The fine-scale stochastic Green's function $\myprime{g} : \mathcal{D} \times \mathcal{D} \times \Omega \times \Omega \rightarrow \mathbb{R}$, which represents the fine-scale Green's operator $\myprime{\mathcal{G}}$, yields the fine-scale component $\myprime{u}$ from the coarse-scale residual. In the finite dimensional case, we can express the action of  $\myprime{g}$ on any arbitrary $\chi \in \mathcal{V}^{*}$, in terms of the deterministic Green's function  $ \myds{g}_{\xi} $, and the $\mu$'s relevant to the projection. Specifically, using \eqref{eqn:fine_greens_finite},
\begin{equation} \label{eqn:G_operator_Prime_chi}
	\myprime{\mathcal{G}}(\chi) = \mathcal{G}(\chi) - \mathcal{G}\bm{\mu}^{T}[\bm{\mu}\mathcal{G}\bm{\mu}^{T}]^{-1}\bm{\mu}\mathcal{G}(\chi).
\end{equation}
We can evaluate the individual terms in \eqref{eqn:G_operator_Prime_chi} using  \eqref{eqn:stochastic_greens_function}, \eqref{eqn:g_gd} and \pref{prop:mu} to obtain the desired expression, {\it i.e.},
\begin{align}
 \label{eqn:fssgf_term1}
	\mathcal{G}(\chi) &=  \int_{\tilde{x} \in \mathcal{D}} \myds{g}_{\xi}(x,\tilde{x})  \chi(\tilde{x},\xi) d\tilde{x},\\
 \label{eqn:fssgf_term2}
	[\mathcal{G} \bm{\mu}^{T}]_{i,m} &= \Phi_m(\xi) \int_{\tilde{x} \in \mathcal{D}} \myds{g}_{\xi}(x,\tilde{x})  \mu_{i}^{d}(\tilde{x}) d\tilde{x},\\
 \label{eqn:fssgf_term3}
	[\bm{\mu}\mathcal{G}\bm{\mu}^{T}]_{i,m,j,n} &= \int_{x \in \mathcal{D}} \int_{\tilde{x} \in \mathcal{D}}   \myds{\mu}_{i}(x) \mathrm{E}\left[ \Phi_m(\xi) \myds{g}_{\xi}(x,\tilde{x}) \Phi_n(\xi)\right] \myds{\mu}_{j}(\tilde{x}) d\tilde{x} dx, \\
 \label{eqn:fssgf_term4}
	[\bm{\mu} \mathcal{G} (\chi) ]_{j,n}  &= \int_{x \in \mathcal{D}}  \int_{\tilde{x} \in \mathcal{D}}  \myds{\mu}_{j}(x) \mathrm{E}\left[ \Phi_n(\xi) \myds{g}_{\xi}(x,\tilde{x}) \chi(\tilde{x},\xi)\right] d\tilde{x} dx.
\end{align} 
\eqref{eqn:G_operator_Prime_chi}$-$\eqref{eqn:fssgf_term4} completely defines the explicit formula for the fine-scale stochastic Green's function. It is the key building block for constructing the VMS formulation for a stochastic PDE. Remarkably, the formula above is expressed entirely in terms of deterministic quantities. In particular the deterministic Green's function $\myds{g}_{\xi}$, and the deterministic functionals $ \mu_{i}^{d}$ that characterize the physical fine-scale space. Therefore we conclude that when it is possible to construct the VMS formulation for a deterministic PDE, then it is also possible to construct one for its stochastic counterpart where the parameters that appear in the differential operator are random variables.

\begin{remark} 
It was pointed out previously that $\myds{\mu}_i = \delta (x -x_i)$, where $x_i$ are coordinates of the physical nodes, are appropriate for deterministic problems posed on one spatial dimension when the VMS method has been defined using an $H^1$ projector. In the case of stochastic problems, such a choice for $\myds{\mu}_i $ will simplify the terms in \eqref{eqn:G_operator_Prime_chi} significantly. In particular,
\begin{align}
	\label{eqn:fssgf_1Dcase_term1}[\mathcal{G} \bm{\mu}^{T}]_{i,m} &= \myds{g}_{\xi}(x,x_i)  \Phi_m(\xi), \\
	\label{eqn:fssgf_1Dcase_term2}[\bm{\mu}\mathcal{G}\bm{\mu}^{T}]_{i,m,j,n} &= \mathrm{E}\bigl[ \Phi_m(\xi) \myds{g}_{\xi}(x_i,x_j) \Phi_n(\xi) \bigr], \\
	\label{eqn:fssgf_1Dcase_term3}[\bm{\mu} \mathcal{G} (\chi) ]_{j,n}  &= \int_{\tilde{x} \in \mathcal{D}}  \mathrm{E}\big[ \Phi_n(\xi) \myds{g}_{\xi}(x_j,\tilde{x}) \chi(\tilde{x},\xi)\big] d\tilde{x}.
\end{align} 
In the following section we use \eqref{eqn:G_operator_Prime_chi}, \eqref{eqn:fssgf_term1}, and \eqref{eqn:fssgf_1Dcase_term1}$-$\eqref{eqn:fssgf_1Dcase_term3}, to determine the fine-scale stochastic Green's function for the advection-diffusion equation in one spatial dimension.

\end{remark}

\section{The stochastic advection diffusion model problem} \label{sec:ade_eg}

In this section we apply the VMS formulation to the stochastic advection-diffusion problem in one spatial dimension. We select $\mathcal{P}$ to be projector defined by the $\mathcal{V}$-inner-product as in \eqref{eqn:projector_def}. We first describe the problem statement and the VMS formulation in \sref{subsec:ade_problem}.  Thereafter, in \sref{subsec:ade_fsgf_localization} we explicitly evaluate the fine-scale stochastic Green's function for this problem and examine its locality in the physical space. We conclude that while the fine-scale stochastic Green's function is non-local, its magnitude outside the element domain under consideration is very small. Motivated by this we approximate the fine-scale stochastic Green's function in each element by the element Green's function, which is identically zero outside the element. We evaluate the performance of the VMS formulation derived from this approximation on problems with one and five random variables in \sref{subsec:numerical_results}.

\subsection{The problem statement} \label{subsec:ade_problem}

As an instance of \eqref{eqn:strong_form}, we consider the advection diffusion problem posed on one physical dimension. We seek the stochastic function $u(x,\xi):\bar{\mathcal{D}} \times \Omega \rightarrow \mathbb{R}$ such that,
\begin{equation}  \label{eqn:ade_strong_form}
     \left. \begin{aligned}
        -\kappa \Delta u(x,\xi) + \beta(x,\xi) \cdot \nabla u(x,\xi) &= f(x,\xi) \qquad x \in \mathcal{D}, \\
         u(x,\xi) &= 0 \qquad \qquad x \in \partial \mathcal{D}.
      \end{aligned}  \right\} 
\end{equation}
where $\bar{\mathcal{D}}=[0 , L]$ is a finite one dimensional domain, $\kappa>0$ is the scalar diffusivity. $f(x,\xi):  \mathcal{D} \times \Omega \rightarrow \mathbb{R} $ is the source term, and $\beta(x,\xi): \mathcal{D} \times \Omega \rightarrow \mathbb{R} $ is the  advection field, both known beforehand. The variational formulation, and the VMS formulation for \eqref{eqn:ade_strong_form} can be constructed using the procedure outlined in \sref{subsec:variational_formulation} and \sref{sec:vms_formulation}. The VMS formulation written using the fine-scale Green's operator $\myprime{\mathcal{G}}$ is: find $\bar{u}(x,\xi) \in \bar{\mathcal{V}}$ such that,
\begin{align}  \label{eqn:ade_weak_form}
\nonumber \mathrm{E}\Bigl[ \int_{x \in \mathcal{D}} & \kappa\nabla\bar{u}  \cdot \nabla\bar{w} dx +  \int_{x \in \mathcal{D}} \beta  \cdot \nabla\bar{u}\bar{w} dx \Bigr] - \mathrm{E}\Bigl[ \int_{x \in \mathcal{D}} \mathcal{L} \myprime{\mathcal{G}} \left( \mathcal{L} \bar{u} \right) \bar{w} dx \Bigr] \\
& = \mathrm{E}\Bigl[  \int_{x \in \mathcal{D}} f  \bar{w} dx  \Bigr] - \mathrm{E}\Bigl[ \int_{x \in \mathcal{D}} \mathcal{L}\myprime{\mathcal{G}}\left(f\right) \bar{w} dx \Bigr], \qquad \forall \bar{w} \in \bar{\mathcal{V}},
\end{align}
where $\mathcal{L} =  -\kappa \Delta + \beta(x,\xi) \cdot \nabla $ is the stochastic advection diffusion differential operator. The Green's function for its deterministic counterpart is well known, \cite{Hu95,HuEtal98}. Hence the fine-scale stochastic Green's function $\myprime{g}$, which defines the fine-scale stochastic Green's operator $\myprime{\mathcal{G}}$, can be constructed using the method detailed in \sref{subsubsec:fssgf}. This is done in the following development.

\subsection{Linear finite elements and localization of the fine-scale space} \label{subsec:ade_fsgf_localization}
\begin{figure}[ht]
  \centering   
  \subcaptionbox{$\mathcal{G}(\chi)$  \label{fig:FSGF_SurfPlot_GoV} }[0.48\linewidth]{\includegraphics[width=.4\linewidth,angle=90]{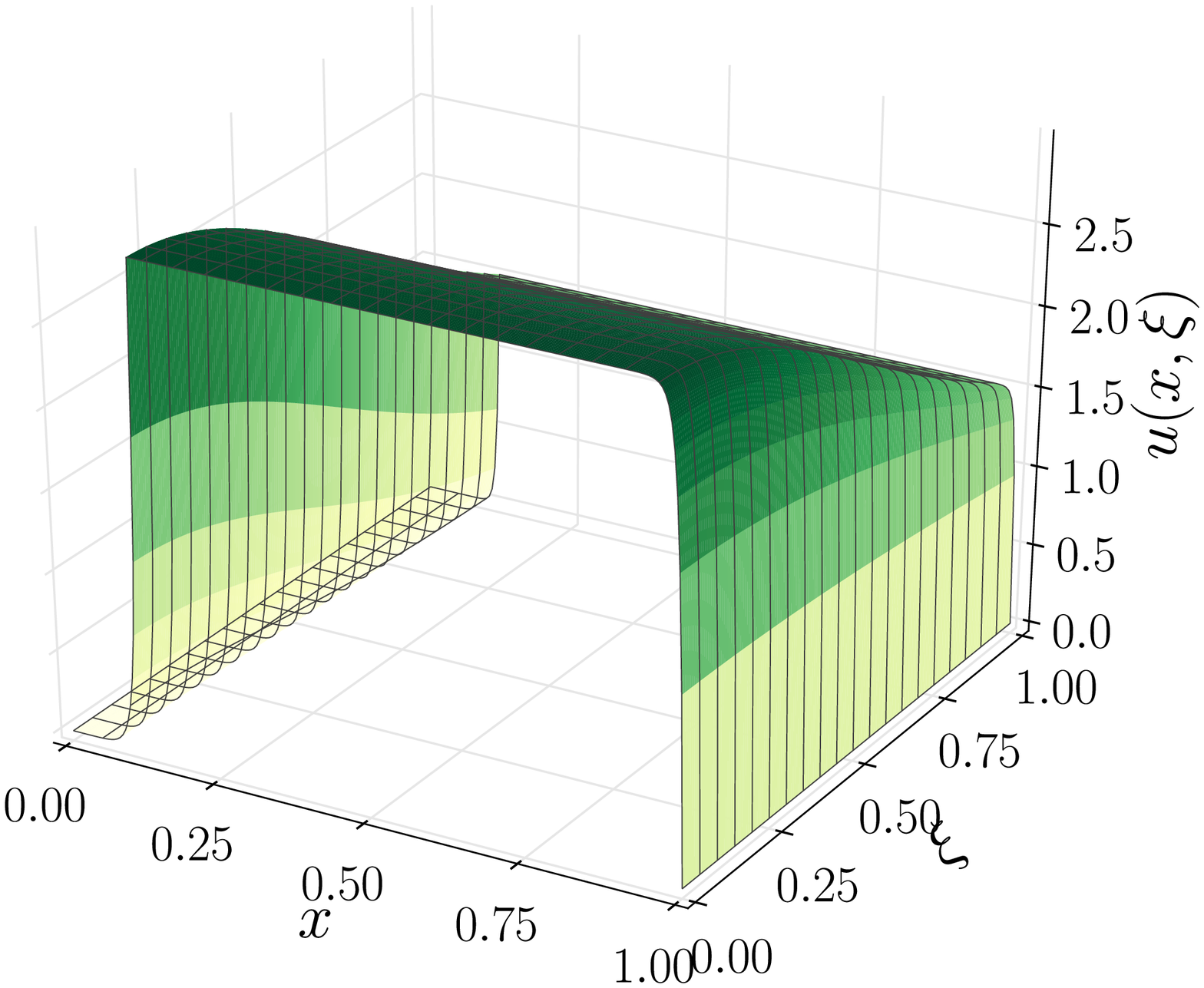}}
  \subcaptionbox{$\myprime{\mathcal{G}}(\chi)$  \label{fig:FSGF_SurfPlot_GFoV} }[0.48\linewidth]{ \includegraphics[width=.4\linewidth,angle=90]{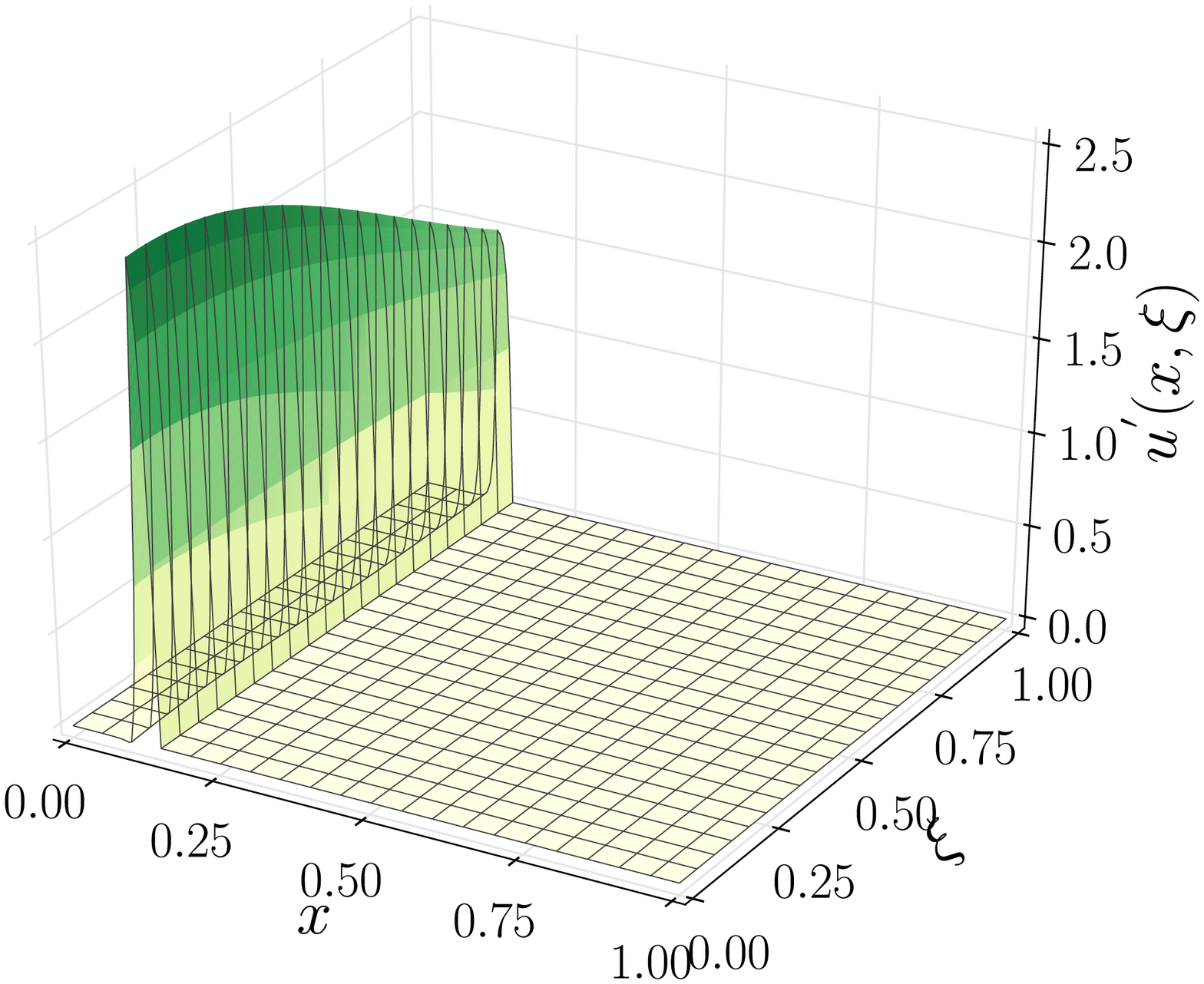} }
 \caption{Comparison of the action of stochastic Green's function and fine-scale stochastic Green's function on a stochastic function $\chi$ specified in \sref{subsec:ade_fsgf_localization}, with $L=1$,  $\kappa=10^{-3}$, a uniform grid of $n_{el}=20$ elements, and a maximum order of gPC $p=2$.} \label{fig:sgf_vs_sfsgf}
\end{figure}
\begin{figure}[ht] 	
  	\centering \includegraphics[width=0.75\textwidth,angle=0]{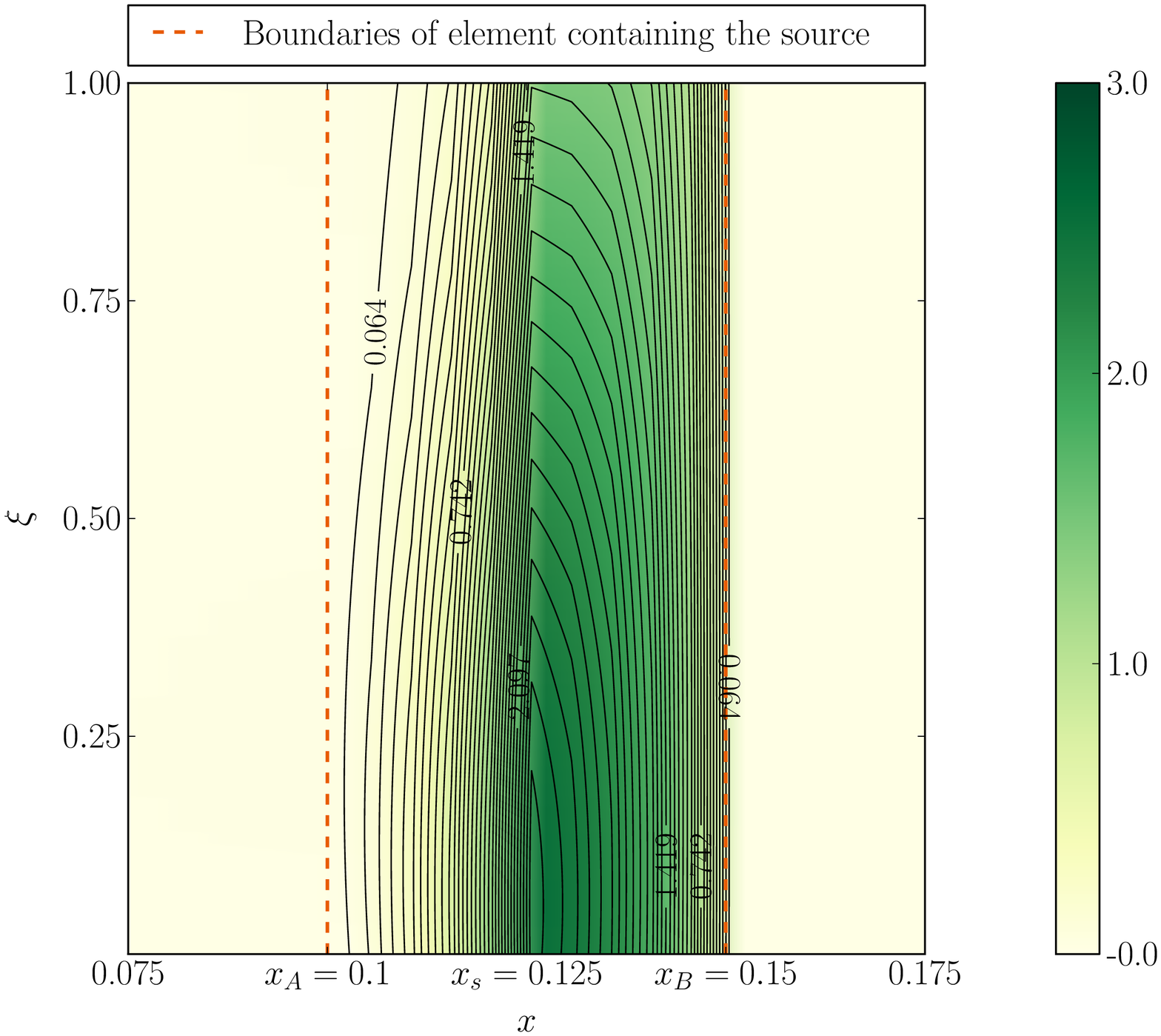}  	                                                          
  	\caption{The fine-scale function $\myprime{\mathcal{G}}(\chi)$ close to the source $x_s$, where $\chi$ and $x_s$ are defined in \sref{subsec:ade_fsgf_localization}.} \label{fig:FSGF_contourPlot_GFoV}
\end{figure}
\begin{figure}[ht] 	
  	\centering \includegraphics[width=0.75\textwidth,angle=90]{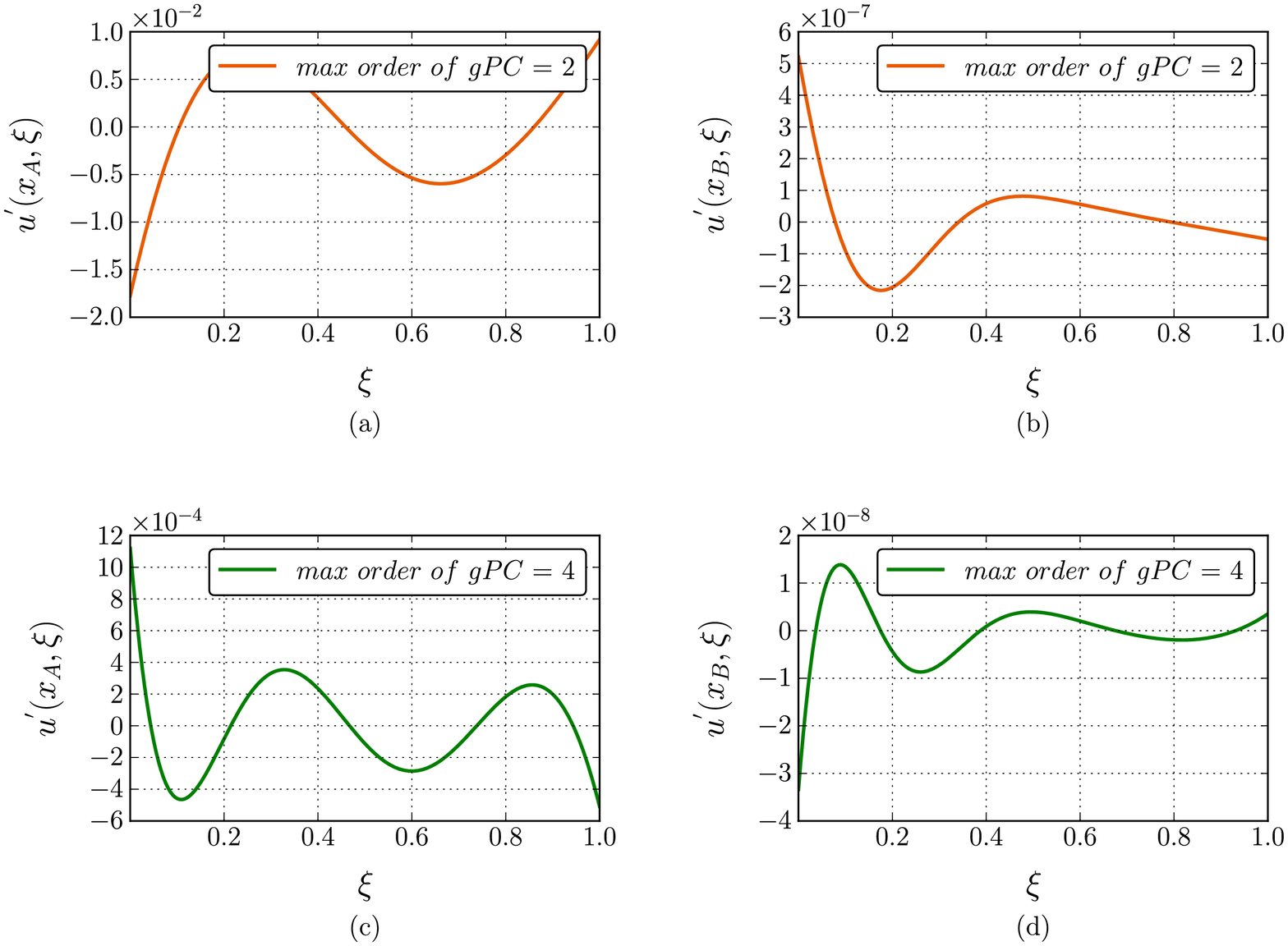}  	                                                          
  	\caption{The fine-scale function at the element boundaries $x_A$ and $x_B$, with the source $\chi$ present within the element. $x_A$, $x_B$ and $\chi$ are defined in \sref{subsec:ade_fsgf_localization}.} \label{fig:FSGF_LinePlot_GFoV_At_EB}
\end{figure}
We consider a physical domain discretized using the grid of nodes,  $0 = x_0 < x_1 < \cdots < x_{n_{el} - 1} < x_{n_{el}} = L$, where $n_{el}$ is the number of elements. An element spans the interval $(x_{i-1}, x_{i})$, where $i=1,\cdots,n_{el}$. We define a linear piecewise polynomial space for the physical discretization, and use gPC of total order less than, or equal to, $p$, to discretize the stochastic dimension. Using this setup, our objective is to examine the locality of the fine-scale Green's function. The stochastic coefficient $\beta(x,\xi)$ is assumed constant in physical space, and a function of a single uniformly distributed random variable. The explicit dependence is chosen to be,
\begin{equation}  \label{eqn:af_def}
\beta(x,\xi) = 1 + \xi^2,  \qquad \xi \sim \mathcal{U}(0,1).
\end{equation}

In  \fref{fig:FSGF_SurfPlot_GoV} and \ref{fig:FSGF_SurfPlot_GFoV} we plot the result of the Green's function and the fine-scale Green's function acting on a source $\chi = \delta(x-x_s)( \phi_0(\xi) + \phi_1(\xi) + \phi_2(\xi))$, where $x_s = 0.125$ is the physical location of the source, and $\phi_i(\xi_i), i =0,1,2$, are the first three Legendre polynomials defined on $[0,1]$. From \fref{fig:FSGF_SurfPlot_GoV} we note that even though $\chi$ is local in physical space, the function $\mathcal{G}(\chi)$ is not. This points to the non-local nature of the Green's operator in the physical space. On the other hand from \fref{fig:FSGF_SurfPlot_GFoV} we observe that the function $\myprime{\mathcal{G}}(\chi)$ appears to be local, in that its value beyond the element which bounds the extent of $\chi$ is very small. This indicates that the fine-scale Green's operator is local in the physical space. It is important to note that in our experience the specific choice of the source $\chi$, advection field $\beta$, and an arbitrary non-uniform mesh, do not alter the local behavior of the fine-scale stochastic Green's function discussed above.

In \fref{fig:FSGF_SurfPlot_GFoV}, the precise spread of the fine-scale function isn't clearly discernible. However, from \fref{fig:FSGF_contourPlot_GFoV}, we observe that even though the resulting fine-scale function is small outside the element, it is non-zero at the boundary of the element containing the source $\chi$, and thus not strictly local in the physical space. In \fref{fig:FSGF_LinePlot_GFoV_At_EB}a and \ref{fig:FSGF_LinePlot_GFoV_At_EB}b we plot the fine-scale function at the element boundaries $x=x_A$, and $x=x_B$, against the range of the random variable. As seen from those plots, the variation of the fine-scale function is higher than quadratic when the underlying gPC have a total order of $p=2$; this is a direct manifestation of Result \ref{result:delta_mu}. If we increase the order to $p=4$, and again plot at the element boundaries $x=x_A$, and $x=x_B$, the fine-scale function resulting with the same source $\chi$,  we observe through \fref{fig:FSGF_LinePlot_GFoV_At_EB}c and \ref{fig:FSGF_LinePlot_GFoV_At_EB}d that the variation is quintic or higher. Additionally, the magnitude of the fine-scale function is significantly lower. This substantiates the earlier claim that in the limit $p \rightarrow \infty$, the fine scales vanish at the nodes, and the fine-scale Green's function converges to the element Green's function.

From the viewpoint of a practical numerical method, the approximation of the fine-scale Green's function by the element Green's function prevents any coupling between the elements in the physical dimension. This ensures that the terms accounting for the effect of fine scales in the coarse-scale VMS formulation, can be computed through the standard finite element assembly procedure and do not alter the stencil of the problem. 

The VMS term that appears in \eqref{eqn:ade_weak_form} simplifies considerably when 
\begin{inparaenum}[(i)] 
\item the fine-scale Green's function in each element is replaced by the element Green's function, that is $\myprime{g}(x,\tilde{x}; \xi,\tilde{\xi}) \approx g^{d,el}_{\xi}(x,\tilde{x}) \delta(\xi - \tilde{\xi})$, where $g^{d,el}_{\xi}$ is the deterministic element Green's function,  
\item it is recognized that for linear finite elements the diffusive term does not contribute within element interiors, and 
\item it is assumed that $\kappa$, $\beta$ and source term $f$ are piecewise constant across the physical domain. 
\end{inparaenum} 
Accounting for these, the VMS term in \eqref{eqn:ade_weak_form}  is simplified as follows,
\begin{align}  \label{eqn:tau_derivation}
        \mathrm{E}\Bigl[ \int_{x \in \mathcal{D}} \mathcal{L} \myprime{\mathcal{G}} \left( r \right) \bar{w} dx \Bigr] 
        & =  \mathrm{E}\Bigl[ \int_{x \in \mathcal{D}} \int_{\tilde{x} \in \mathcal{D}} \int_{ \tilde{\xi} \in \Upsilon } \myprime{g}(x,\tilde{x}; \xi,\tilde{\xi})  r\bigl(\tilde{x},\tilde{\xi} \bigr) \mathcal{L}^{*} \bar{w}(x,\xi)  d\tilde{x} d\tilde{\xi} dx\Bigr], \\
        & \approx \mathrm{E}\Bigl[  \sum_{i=1}^{n_{el}} \int_{x_{i-1}}^{x_{i}} \int_{x_{i-1}}^{x_{i}} \int_{ \tilde{\xi} \in \Upsilon } g^{d,el}_{\xi}(x,\tilde{x}) \delta(\xi - \tilde{\xi})  r\bigl(\tilde{x},\tilde{\xi} \bigr) \mathcal{L}^{*} \bar{w}(x,\xi)  d\tilde{x} d\tilde{\xi} dx\Bigr], \\
       \label{eqn:vms_contrib} & = \mathrm{E}\Bigl[  \sum_{i=1}^{n_{el}} \int_{x_{i-1}}^{x_{i}} \int_{x_{i-1}}^{x_{i}}  g^{d,el}_{\xi}(x,\tilde{x})   r\bigl(\tilde{x},\xi \bigr) \mathcal{L}^{*} \bar{w}(x,\xi)  d\tilde{x}  dx\Bigr], \\
        = &   \sum_{i=1}^{n_{el}}  \mathrm{E} \Bigl[  \biggl(  \frac{\int_{x_{i-1}}^{x_{i}} \int_{x_{i-1}}^{x_{i}} g^{d,el}_{\xi}(x,\tilde{x}) dx d\tilde{x}}{h_i} \biggr) h_i r\bigl(\tilde{x},\xi \bigr) \mathcal{L}^{*} \bar{w}(x,\xi) \Bigr], \\
         \label{eqn:tau_g_relation} & =   \sum_{i=1}^{n_{el}}  \mathrm{E} \Bigl[   \tau_i(\xi) h_i r\bigl(\xi \bigr) \mathcal{L}^{*} \bar{w}(\xi)  \Bigr],        
\end{align}
where $r = f - \mathcal{L}\bar{u}$ is the coarse-scale residual, $ \mathcal{L}^{*} =   -\kappa \Delta - \beta(x,\xi) \cdot \nabla $  is the adjoint of $\mathcal{L}$, $ h_i =  x_{i} - x_{i-1} $ is the size of the $i^{th}$ element, and the term 
$\tau_i(\xi)$  is recognized as the stabilization parameter $\tau$ \cite{Hu95}. In this context however, it is necessary to note that $\tau$ is a stochastic function, which is piecewise constant across the physical domain. In particular,
\begin{align}  \label{eqn:tau_def}
        \tau_{i} (\xi)  =  \frac{\int_{x_{i-1}}^{x_{i}} \int_{x_{i-1}}^{x_{i}} g^{d,el}_{\xi}(x,\tilde{x}) d\tilde{x} dx}{h_i} = \frac{h_i}{2\beta}\left( \coth(Pe(\xi)) - \frac{1}{Pe(\xi)}\right),
\end{align}
where $Pe(\xi) = h\beta(\xi)/2\kappa$ is the mesh Pecl\'{e}t number.

\paragraph*{Remarks}
\begin{enumerate}
\item We note that the final expression \eqref{eqn:tau_g_relation} is valid only when it is assumed that the contribution from the residual and the weighting function is constant (but random) over an element. In the case of advection-diffusion equation this is true only for linear finite elements with piecewise-constant problem parameters. 
\item On the other hand, expression \eqref{eqn:vms_contrib} is valid for any variation of the residual and weighting functions. In that sense, \eqref{eqn:vms_contrib} is a more general expression. In cases where the variation of the residual and the contribution from the weighting functions is known to be of a certain polynomial order, the double integral in the physical space can be simplified by expanding $\tau$ as a polynomial function in $x$ and $\tilde{x}$ (see \cite{Oberai98} for example), and evaluating the integral analytically. 
\item The VMS solution is guaranteed to yield coarse-scale solutions that are optimal in the sense defined by the projector $\mathcal{P}$, for instance as defined in \eqref{eqn:projector_def}, if the exact expression for the fine-scale Green's function is used. With our choice of the projector in one dimension this translates to the property that the coarse-scale solution will have exact gPC coefficients at the physical nodal coordinates (\sref{prop:projector}). When we approximate the fine-scale Green's function with the element Green's function, this is no longer guaranteed. However, since we know this to be a reasonable approximation, we might expect the gPC coefficients at the nodes to be quite close to the exact values. This hypothesis is numerically verified in the following section for problems posed in a single and in five stochastic dimensions.
\end{enumerate}
\subsection{Numerical results} \label{subsec:numerical_results}
\subsubsection{A single random variable example} \label{subsubsec:1RV_results}

\begin{figure}[ht]
	\centering
	\subcaptionbox{Analytical solution \label{fig:1RV_surfPlot_ExactSoln}}[0.49\linewidth]{\includegraphics[width=0.4\textwidth,angle=90]{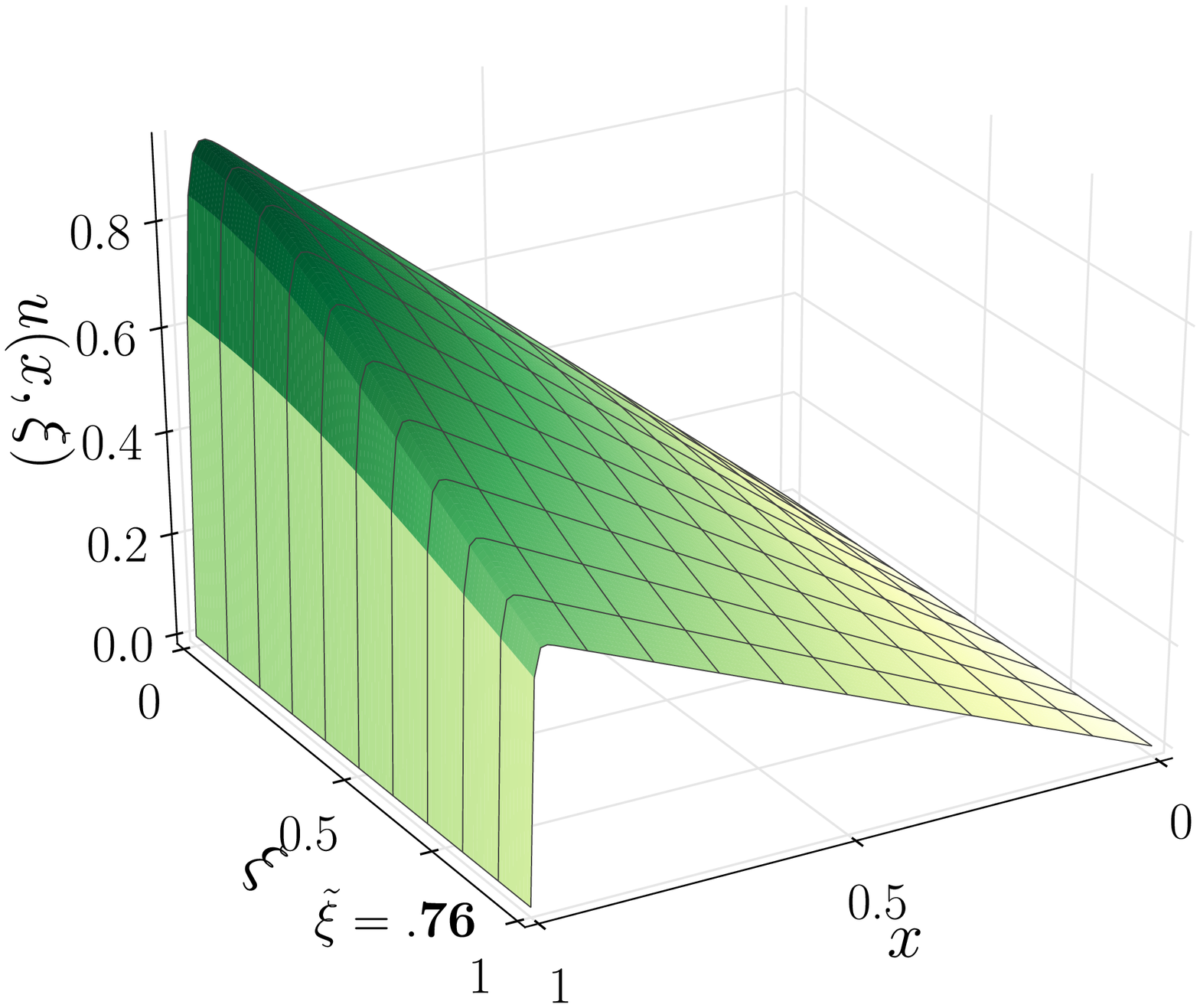}}
	\subcaptionbox{Realization of the Galerkin and VMS solutions using a uniform grid of $n_{el}=20$ elements, and a maximum order of gPC $p=2$.\label{fig:1RV_LinePlot_GalerkinVMS_realization}}[.49\linewidth]{\includegraphics[width=0.4\textwidth,angle=90]{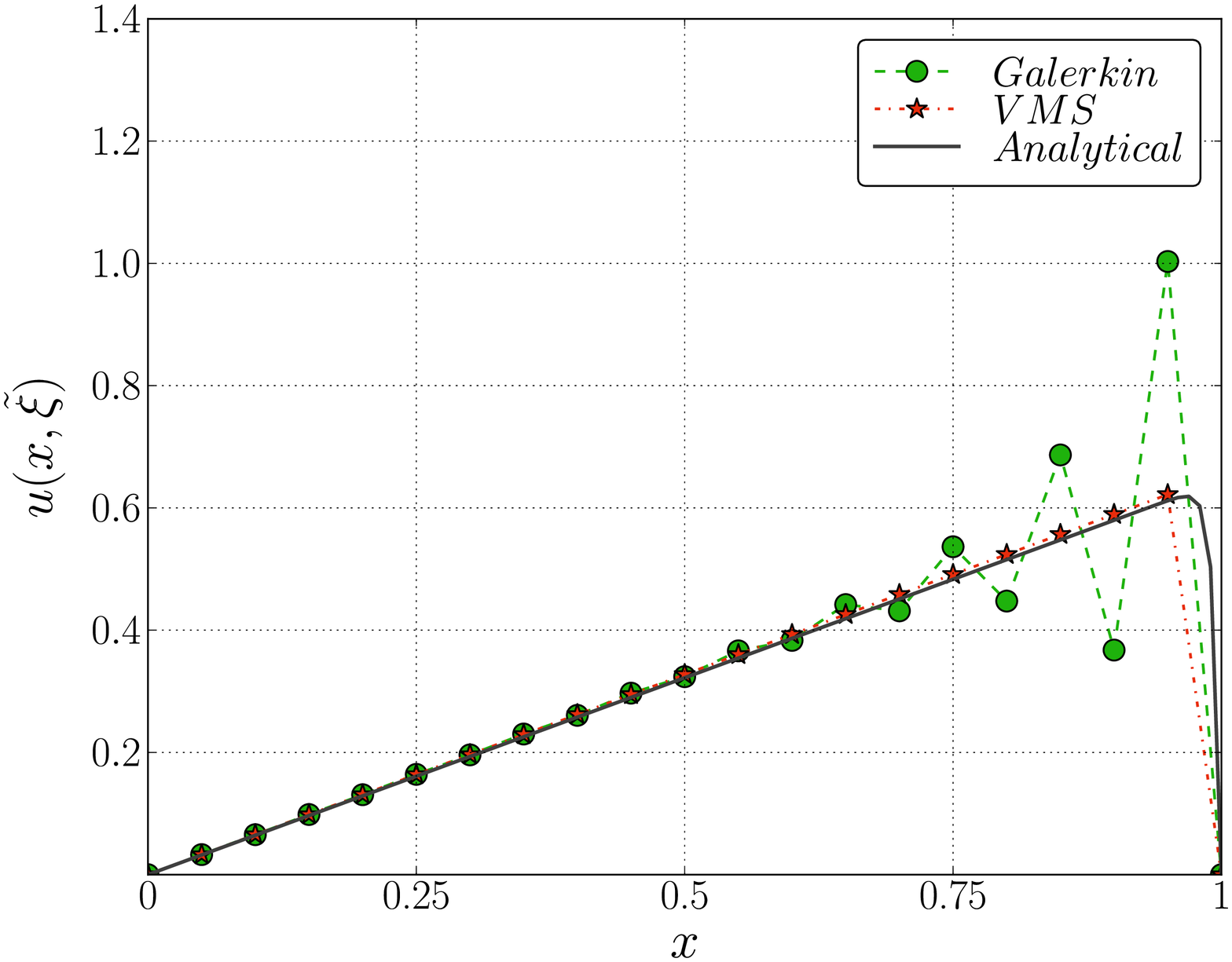}}
	\caption{The case of single random variable, with $L=1$,  $\kappa=10^{-3}$, and $f=1$.} 
	\label{fig:1RV_realization}
\end{figure}

If the advection field $\beta(x,\xi)$ in \eqref{eqn:ade_strong_form} is defined identical to \eqref{eqn:af_def}, it is possible to obtain a closed form solution for the stochastic function $u(x,\xi)$. In \fref{fig:1RV_surfPlot_ExactSoln} we plot the analytical solution for the case when the source term $f$ is a constant with no uncertainty. The realizations of the standard Galerkin and VMS solutions for the instance $\tilde{\xi}$ of the random variable $\xi$ is shown in \fref{fig:1RV_LinePlot_GalerkinVMS_realization}. Unlike the Galerkin method, the VMS solution doesn't exhibit any oscillations and is noticeably better. The mean and variance of the coarse-scale Galerkin and VMS solutions at each interior physical node are shown in \fref{fig:1RV_LinePlot_Mean} and \ref{fig:1RV_LinePlot_Variance}, respectively. We observe that in addition to producing a solution that is accurate for a single draw, the VMS solution produces accurate statistics of the solution. In  \fref{fig:1RV_linePlot_Error_OBSolnCof} we compare the errors in the gPC coefficients of the coarse-scale Galerkin and VMS solutions evaluated at the physical nodal points. We observe that the error in the coefficients for the VMS solution is smaller for most coefficients. In particular the maximum error in any of the coefficients for the VMS solution is about two orders of magnitude smaller than that for the Galerkin solution.

\begin{figure}[ht]
  \centering 
  \subcaptionbox{Mean \label{fig:1RV_LinePlot_Mean}}[0.49\linewidth]{\includegraphics[width=.4\linewidth,angle=90]{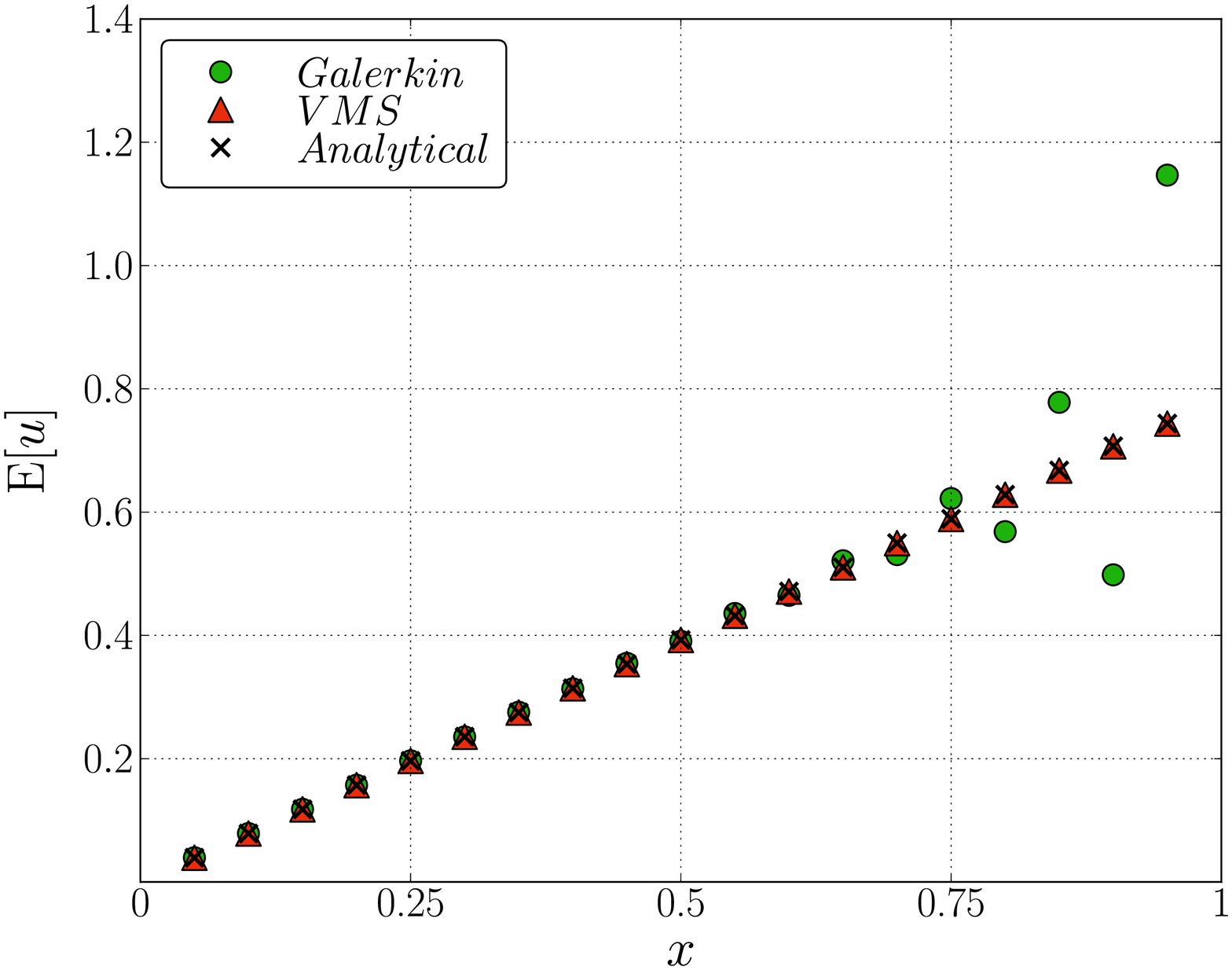}  }
  \subcaptionbox{Variance \label{fig:1RV_LinePlot_Variance} }[0.49\linewidth] {\includegraphics[width=.4\linewidth,angle=90]{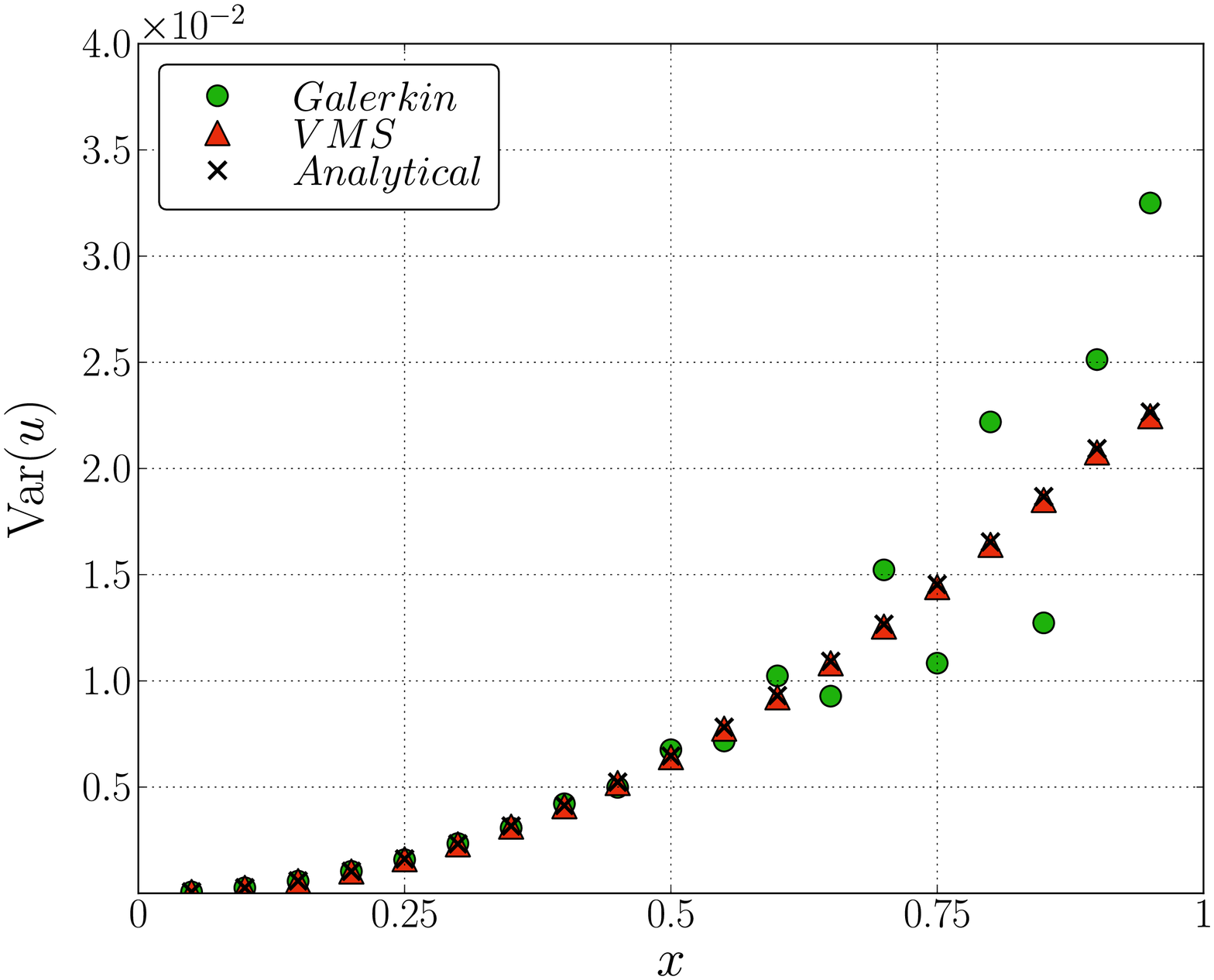}  }
   \caption{Statistical moments of the Galerkin and VMS solutions for the single random variable example.}
   \label{fig:1RV_stats}
\end{figure}

\begin{figure}[ht] 	
  	\centering \includegraphics[width=0.75\textwidth,angle=90]{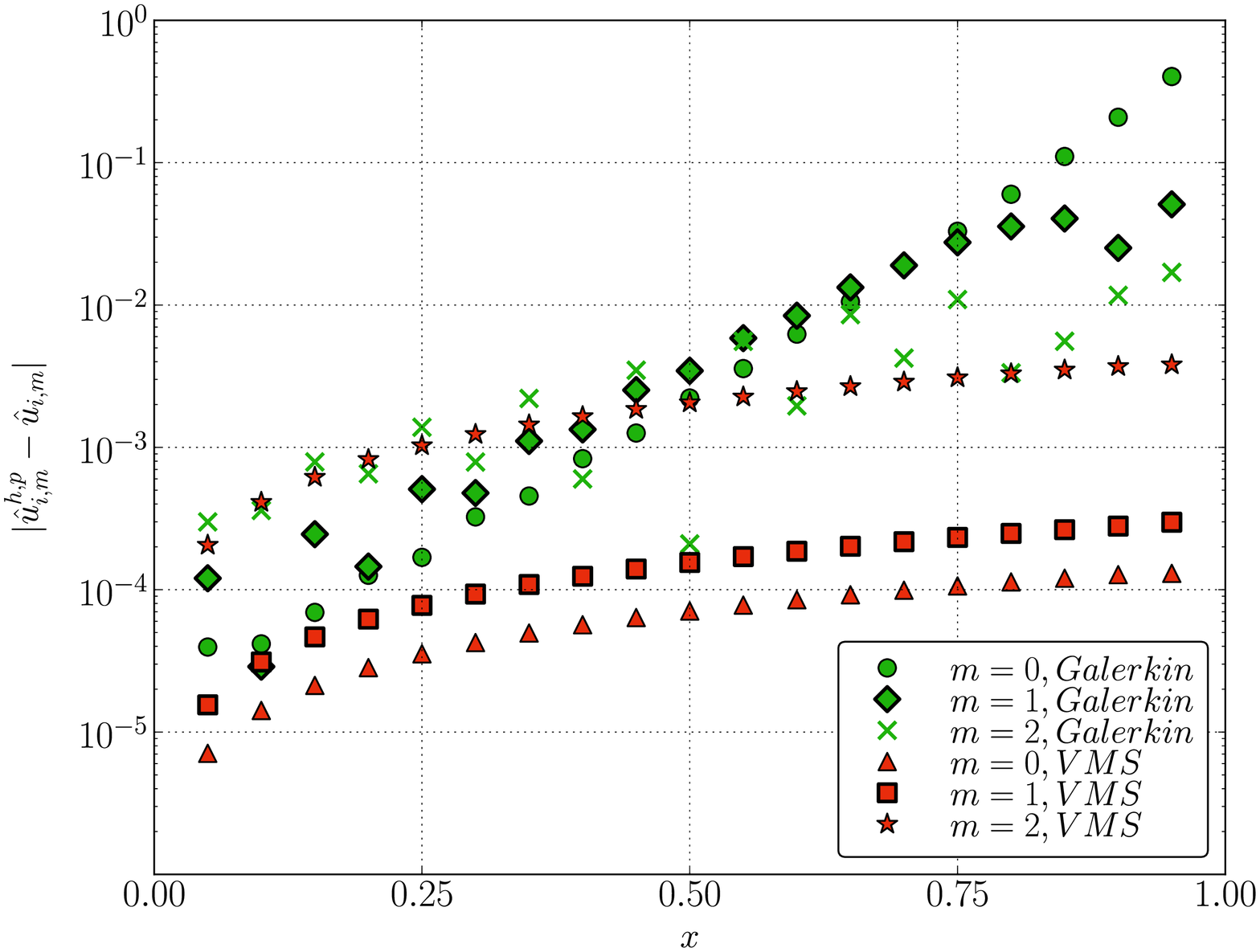}  	                                                          
  	\caption{Error of the gPC coefficients at nodal locations for the Galerkin and VMS method for the single random variable example.}
  	\label{fig:1RV_linePlot_Error_OBSolnCof}
\end{figure}

\subsubsection{A problem with multiple random variables} \label{subsubsec:nRV_results}

\begin{figure}[ht]
\centering 
  \subcaptionbox{Realization of the advection field \label{fig:nRV_LinePlot_advFieldRealization}}[0.49\linewidth]{\includegraphics[width=0.4\textwidth,angle=90]{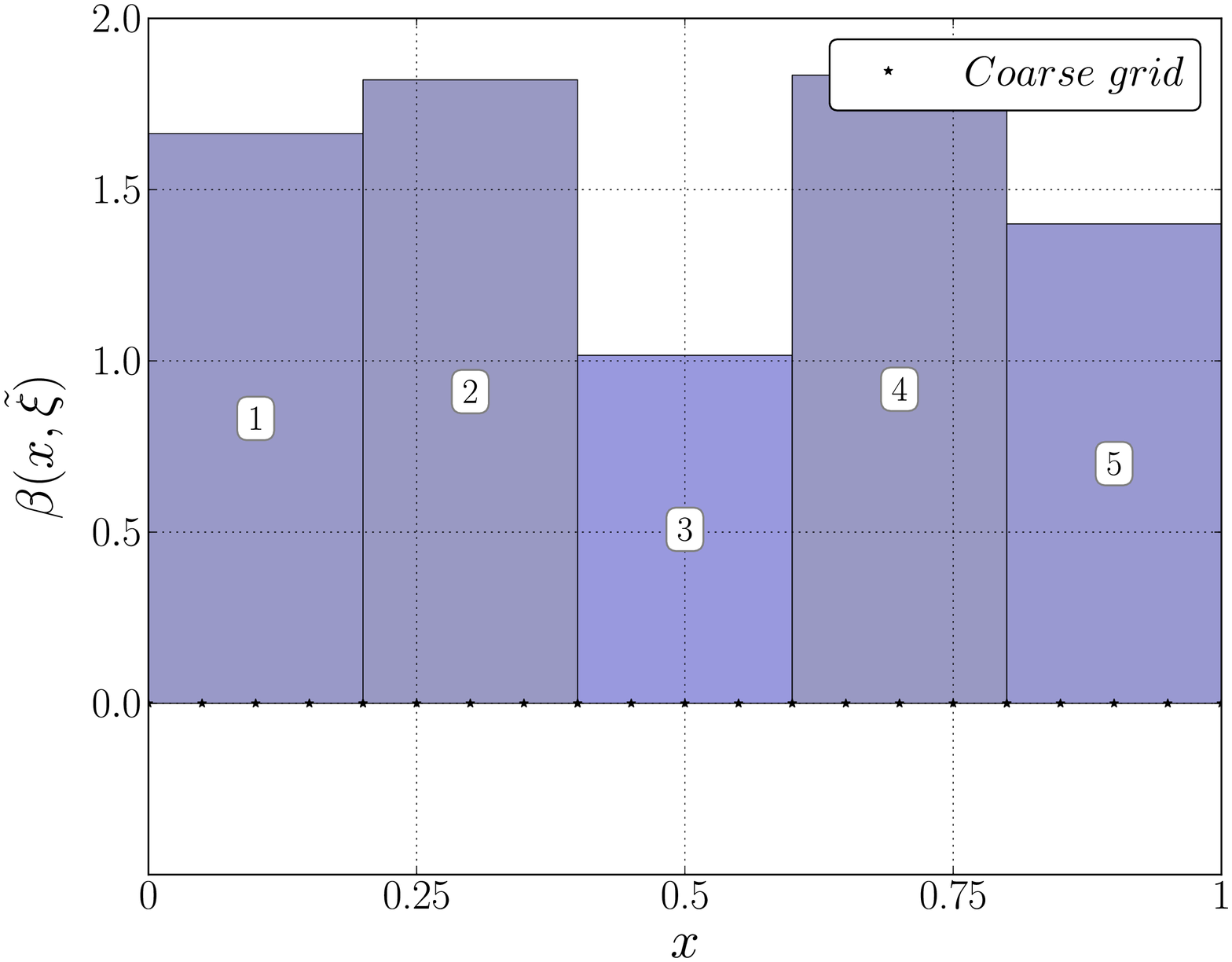} }
  \subcaptionbox{Realzation of the Galerkin and VMS solutions  \label{fig:nRV_LinePlot_Realization}}[0.49\linewidth]{\includegraphics[width=.4\linewidth,angle=90]{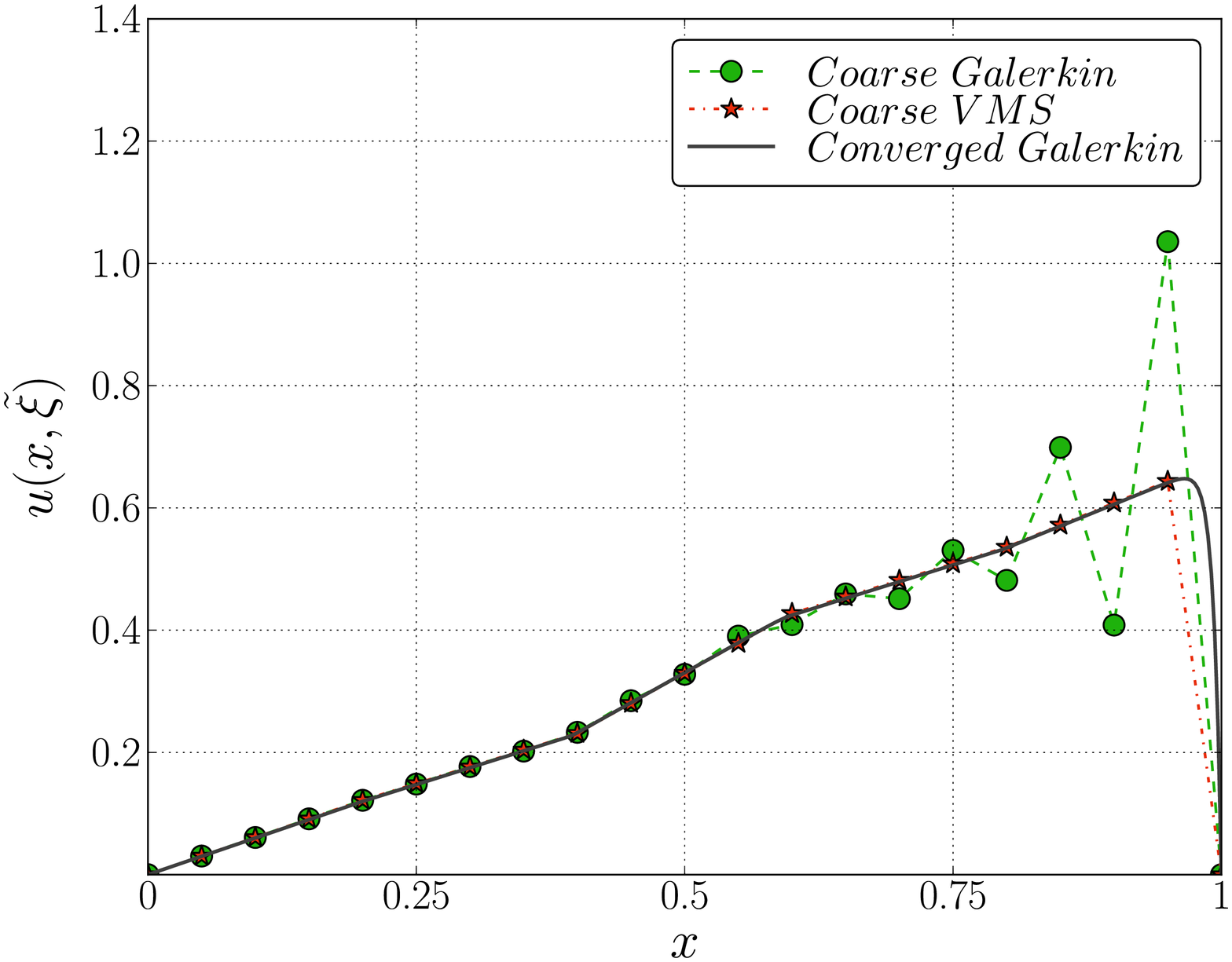}  }
  \caption{A multiple random variable example  with $L=1$,  $\kappa=10^{-3}$, $f=1$, a uniform grid of $n_{el}=20$ elements, and a total order of gPC $p=2$.}
  \label{fig:nRV_realization}
\end{figure}

\begin{figure}[ht]
\centering 
  \subcaptionbox{Realization of each discretization. \label{fig:nRV_LinePlot_Galerkin_Convergence} }[0.49\linewidth]{\includegraphics[width=0.4\textwidth,angle=90]{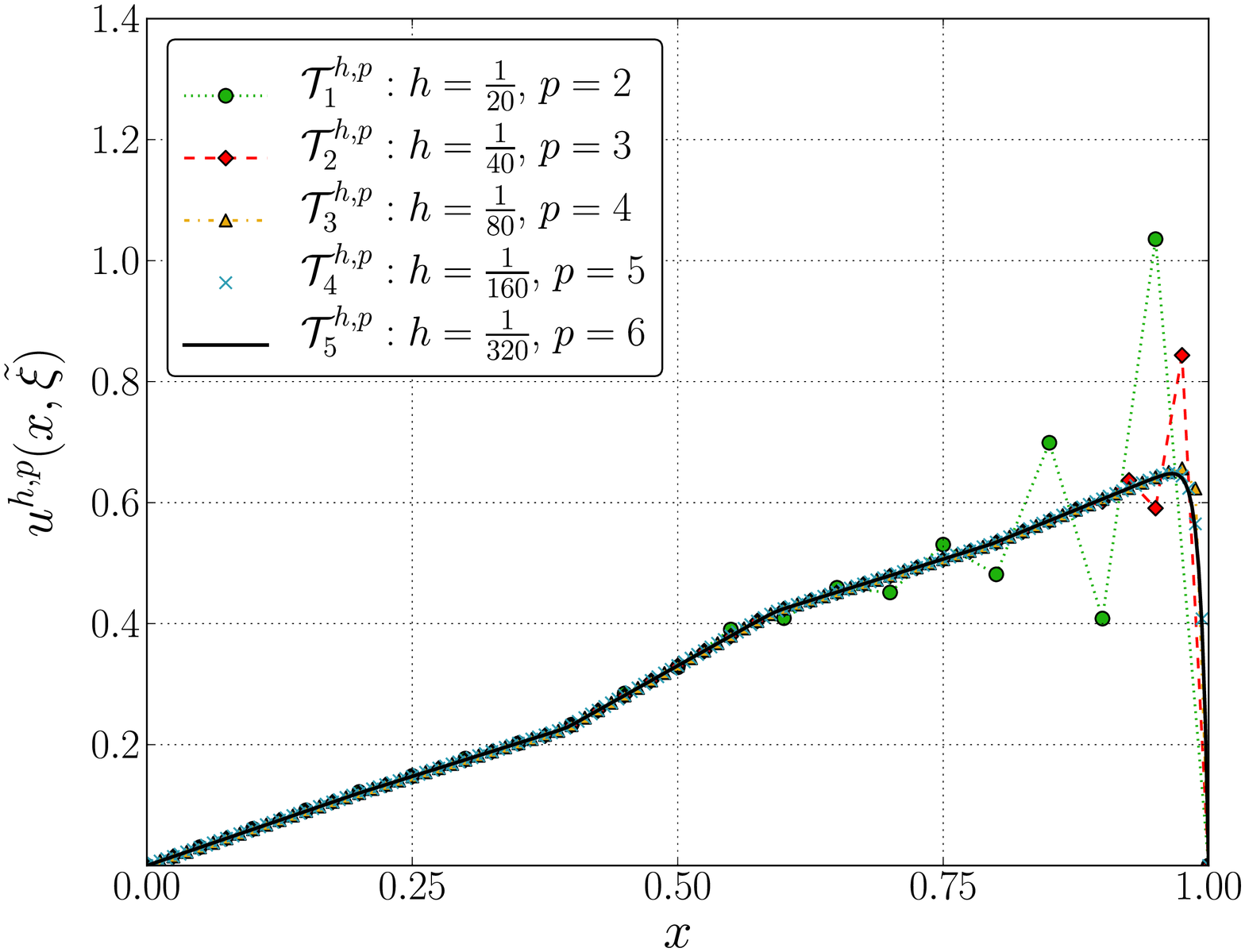} }
  \subcaptionbox{ Distance between successive discretizations measured in the norm of $\mathcal{V}$. \label{fig:nRV_H1Error_CS}}[0.49\linewidth]{\includegraphics[width=.4\linewidth,angle=90]{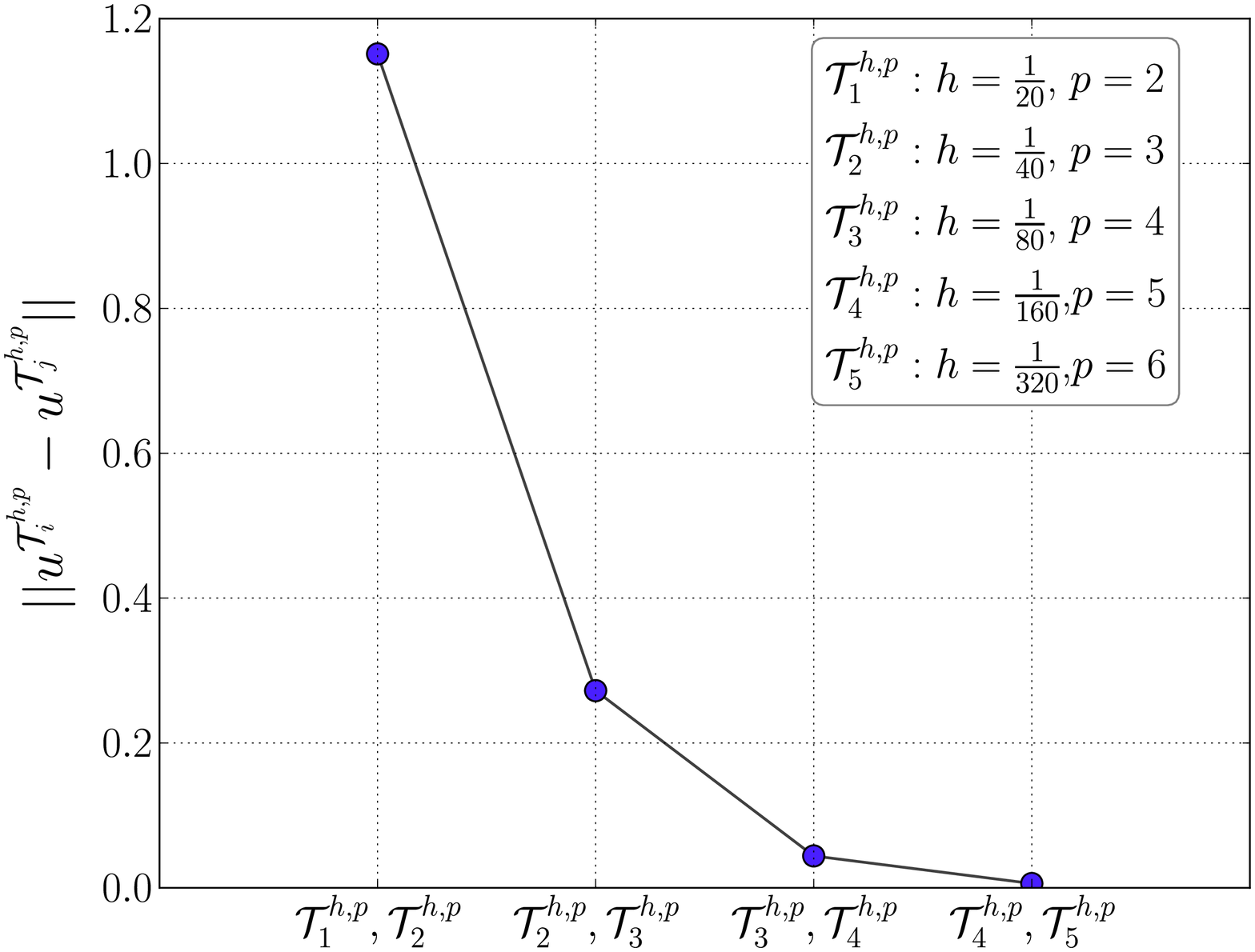}  }
  \caption{Convergence of the standard Galerkin method for the multiple random variable example.}
  \label{fig:nRV_convergence}
\end{figure}

\begin{figure}[ht]
  \centering 
  \subcaptionbox{Mean \label{fig:nRV_LinePlot_Mean}}[0.49\linewidth]{\includegraphics[width=.4\linewidth,angle=90]{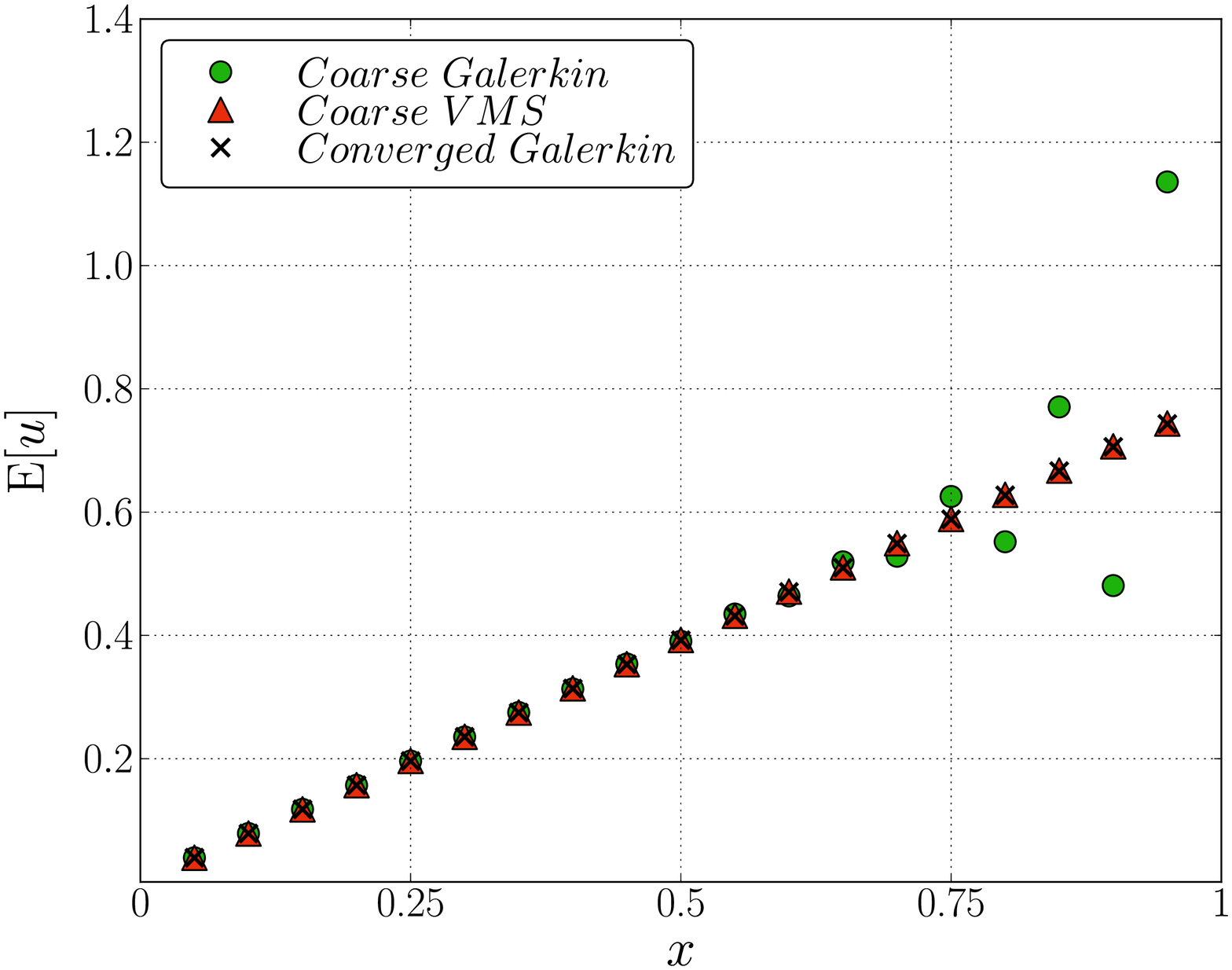}  }
  \subcaptionbox{Variance \label{fig:nRV_LinePlot_Variance}}[0.49\linewidth]{\includegraphics[width=.4\linewidth,angle=90]{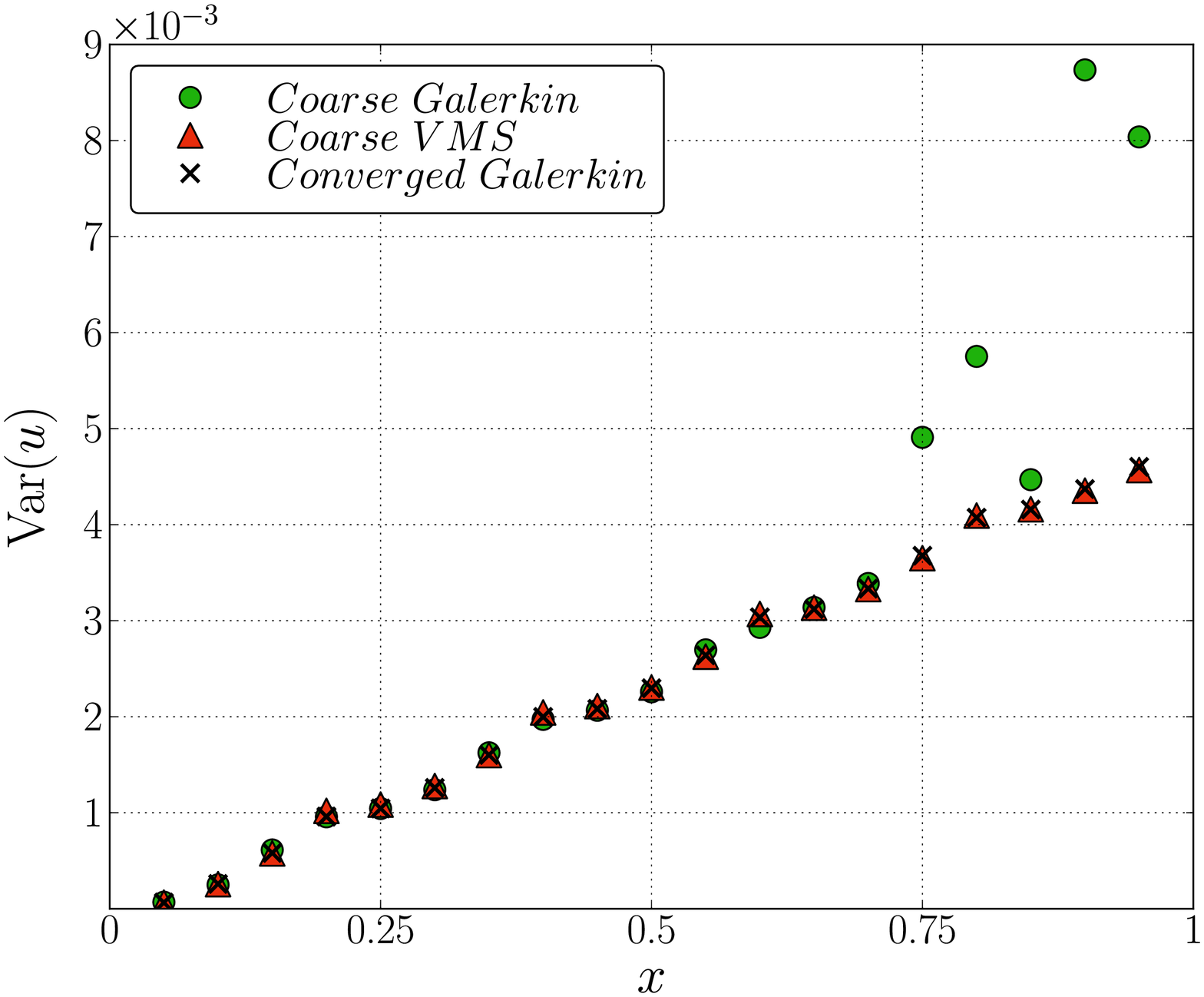}  }
  \caption{Statistical moments of the Galerkin and VMS solutions for the multiple random variable example.}
  \label{fig:nRV_stats}
\end{figure}

\begin{figure}[ht] 	
  	\centering \includegraphics[width=1.00\textwidth,angle=0]{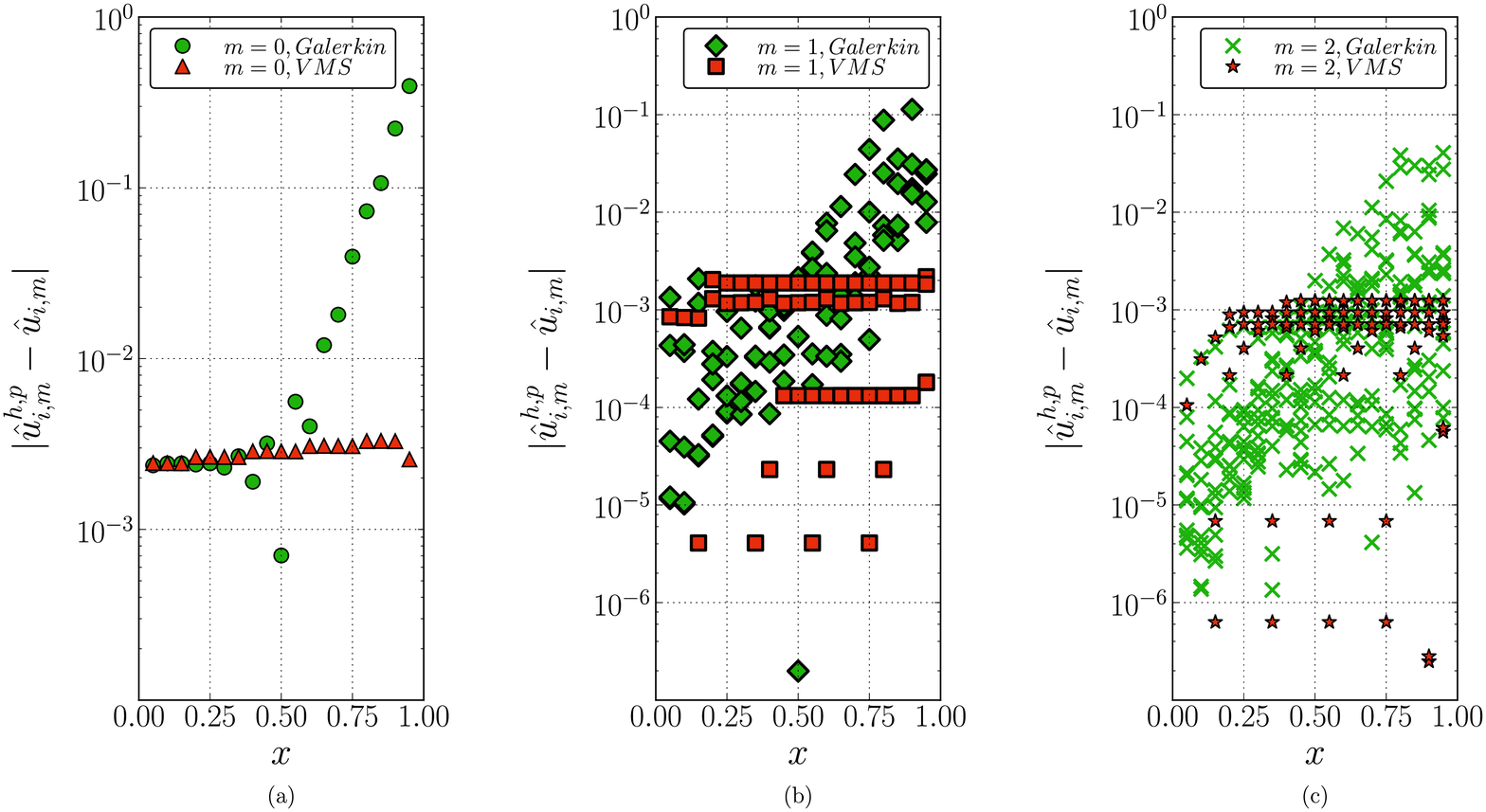}  	                                                          
  	\caption{Error in the nodal gPC coefficients of the Galerkin and VMS method for the multiple random variable example.  The error for a large majority of the coefficients of the VMS method aren't depicted in the plot, since they are significantly smaller and outside the limits of the ordinate.}
  	\label{fig:nRV_LinePlot_OBSolnCof}
\end{figure}

We consider the scenario when the advection field $\beta(x,\xi)$ is a function of multiple random variables. It is defined such that, in a particular region $(X^{min}_{k},X^{max}_{k})$ of the physical domain $\bar{\mathcal{D}}$,  it depends solely on the random variable $\xi_k$ which is assumed to be uniformly distributed in the interval $(0,1)$.  The random variables $\xi_k$ corresponding to different regions $k$ are assumed to be statistically independent. The exact functional dependence of the advection field on each $\xi_k$ is chosen identical to  \eqref{eqn:af_def}, {\it i.e.},
\begin{equation} 
\beta(x,{\xi}) = 
\begin{cases}
     1 + {\xi}_{1}^2 , \qquad & 0 \leq x < X^{max}_{1},  \quad \qquad \xi_1 \sim \mathcal{U}(0,1),\\
    \cdots \cdots \\
    1 + {\xi}_{k}^2 , \qquad & X^{min}_{k} \leq x < X^{max}_{k}, \quad \xi_k \sim \mathcal{U}(0,1),\\
    \cdots \cdots \\
    1 + {\xi}_{q}^2 , \qquad & X^{min}_{q} \leq x \leq L , \quad \qquad \xi_q \sim \mathcal{U}(0,1).
\end{cases} 
\end{equation}

In \fref{fig:nRV_LinePlot_advFieldRealization}, we plot the realization of the advection field for the instance $\tilde{\xi}$ of the random vector $\xi$. As is apparent from the plot, there are five random variables and the advection field is discontinuous across the physical domain. Such a situation can arise for instance, in species transport through random porous media with independent layers. Considering once again the case of a constant source term, and for the same instance $\tilde{\xi}$, in  \fref{fig:nRV_LinePlot_Realization} we plot the  Galerkin and VMS solutions obtained using the coarse scales, along with the converged Galerkin solution. Once again we observe that the VMS solution is almost nodally exact, whereas the Galerkin solution (for the same total order of gPC) is polluted by spurious oscillations.

The veracity of the converged Galerkin solution is established using a refinement study and is depicted in \fref{fig:nRV_convergence}. In particular, \fref{fig:nRV_LinePlot_Galerkin_Convergence} is a plot of the realization of the standard Galerkin solution for the instance  $\tilde{\xi}$ for five different discretizations, and in \fref{fig:nRV_H1Error_CS} we plot the distance between successively finer discretizations as measured in the $\mathcal{V}$ norm. It is apparent that the Galerkin method converges, and henceforth the solution obtained from the discretization $\mathcal{T}^{h,p}_5, h=\frac{1}{320},p=6$ will be referred to as the converged solution.

In \fref{fig:nRV_LinePlot_Mean} and \ref{fig:nRV_LinePlot_Variance} we compare the mean and the variance, respectively, of the coarse-scale Galerkin and VMS solutions with that of the converged solution. It is evident that even for this multi-dimensional example the statistics of the numerical solution are better captured with the VMS method. It is remarkable that they are almost nodally exact. 

In \fref{fig:nRV_LinePlot_OBSolnCof} we plot the error of the individual gPC coefficients of the coarse-scale Galerkin and VMS solution; the error being computed with respect to the converged solution. It is essential to point out that, the error for a large majority of the coefficients of the VMS method aren't depicted in the plot, since they are significantly smaller and outside the limits of the ordinate. At each order of the gPC, the maximum error in the VMS coefficients is around two orders of magnitude smaller that the corresponding Galerkin coefficients.

It is interesting to consider the drastic reduction in the size of the problem obtained by adopting the VMS method. The converged Galerkin solution obtained using the discretization  $\mathcal{T}^{h,p}_5, h=\frac{1}{320},p=6$ gives rise to a linear system with $147,378$ unknowns. In comparison, the VMS solution (which is satisfactory for computing the mean, variance and the gPC coefficients up to the second order) is obtained with the coarse-scale discretization $\mathcal{T}^{h,p}_1, h=\frac{1}{20},p=2$ contains $399$ unknowns, which is smaller by a factor of $\approx 370$. In our view this makes a compelling argument to include the VMS method in stochastic Galerkin solution approaches.

\section{Conclusions} \label{sec:conclusions}
We have presented the variational multiscale method for PDEs with stochastic coefficients and source terms. We have used it as a method for generating accurate coarse-scale solutions while accounting for the effect of the missing scales through a model term that contains a fine-scale stochastic Green's function.

We have analyzed the fine-scale Green's function for stochastic PDEs, and demonstrated that it is closely related to its deterministic counterpart when the projector used to define the ``optimal'' coarse-scale solution is induced by the natural inner-product of the space of trial-solutions, and the basis functions for the space of random variables are $L^2$-orthogonal.  In particular, we have
\begin{inparaenum}[(i)] 
\item shown that at the physical locations where the fine-scale deterministic Green's function vanishes, the fine-scale stochastic function satisfies a weaker, and discrete notion of vanishing stochastic coefficients, and 
\item derived an explicit formula for the fine-scale stochastic Green's function that only involves quantities needed to evaluate the fine-scale deterministic Green's function.
\end{inparaenum}
By making use of the first observation, we have argued that the fine-scale Green's function may be reasonably approximated by the element Green's function. This results in a tremendous simplification in the VMS formulation of stochastic PDEs. We have tested the consequence of making this approximation for the stochastic advection-diffusion problem and noted that the resulting VMS method is more accurate than the standard stochastic Galerkin approach.

We note that there are several interesting ideas to pursue that are related to the topic of this manuscript. First, the VMS integral involving $\tau$, the stabilization parameter in \eqref{eqn:tau_g_relation}, which is a complicated function of the stochastic parameters, can be computationally challenging. Thus approximation for $\tau$ that are based on a polynomial expansion might be very useful in lowering these costs. Second, the VMS method provides an estimate for $u'$, which can be thought of as an approximation to the error in the coarse-scale solution. An interesting avenue of research would be to develop an adaptive (in physical and stochastic spaces) refinement strategy based on this error indicator \cite{Hauke06}.

\section{Acknowledgements}
Assad Oberai and Jayanth Jagalur-Mohan gratefully acknowledge financial support by the CCNI/IBM research projects.  Onkar Sahni gratefully acknowledges the financial support of the Department of Energy, Office of Science's SciDAC-III Institute, as part of the Frameworks, Algorithms, and Scalable Technologies for Mathematics (FASTMath) program, under grant DE-SC0006617. Alireza Doostan gratefully acknowledges the financial support of the Department of Energy under Advanced Scientific Computing Research Early Career Research Award DE-SC0006402.

\bibliographystyle{plain}
\bibliography{bibtex_citations}
\end{document}